\documentclass[11pt]{amsart}
\usepackage{epsfig,amsmath,latexsym,amssymb,scalefnt,url,graphicx,amsfonts,amsthm,color,bm,multirow}

\usepackage[colorlinks=true]{hyperref}
\usepackage{hyperref}
\usepackage{breakurl}

\newtheorem{theorem}{Theorem}[section]
\newtheorem{proposition}{Proposition}[section]

\newtheorem{definition}{Definition}[section]

\newtheorem{lemma}{Lemma}[section]

\graphicspath{{./img/}}

\usepackage{color,bm}
\date{\today}

\title[Discrete representations for flow topology]{Discrete representations of orbit structures  of flows for topological data analysis}

\author[Takashi Sakajo and Tomoo Yokoyama]{}
 

 \email{sakajo@math.kyoto-u.ac.jp}
 \email{tomoo@gifu-u.ac.jp}
 
\thanks{T.S. is supported by JST-Mirai Program Grant No. JPMJMI18G3 and JPMJMI22G1. T.Y. is partially supported by JSPS Kakenhi(C) (\#20K03583). This work was partially supported by a grant from the Simons Foundation. T.S. would like to thank the Isaac Newton Institute for Mathematical Sciences for support and hospitality during the programme [CAT] when work on this paper was undertaken. This work was supported by: EPSRC grant number EP/R014604/1. On behalf of all authors, the corresponding author states that there is no conflict of interest.}

\newcommand{\B}{\Box}

\begin{document}
\maketitle

\centerline{\scshape Takashi Sakajo}
\medskip
{\footnotesize
 \centerline{Department of Math, Kyoto University, Kyoto 606-8602, Japan}
} 

\medskip

\centerline{\scshape Tomoo Yokoyama}
\medskip
{\footnotesize
 \centerline{Applied Mathematics and Physics Division, Gifu University, Yanagido 1-1, Gifu, 501-1193, Japan}
}

\bigskip

\begin{abstract}
This paper shows that the topological structures of particle orbits generated by a generic class of vector fields on spherical surfaces, called {\it the flow of finite type},  are in one-to-one correspondence with discrete structures such as trees/graphs and sequence of letters.  The flow of finite type is an extension of structurally stable Hamiltonian vector fields, which appear in many theoretical and numerical investigations of 2D incompressible fluid flows. Moreover, it contains compressible 2D vector fields such as the Morse--Smale vector fields and the projection of 3D vector fields onto 2D sections. The discrete representation is not only a simple symbolic identifier for the topological structure of complex flows, but it also gives rise to a new methodology of topological data analysis for flows when applied to data brought by measurements, experiments, and numerical simulations of complex flows. As a proof of concept, we provide some applications of the representation theory to 2D  compressible vector fields and a 3D vector field arising in an industrial problem.
 \end{abstract}

\section{Introduction}
{Fluid dynamics is one of the important fields of science and technology in the modern era. Large-scale numerical simulations of process-driven PDE models such as the Navier-Stokes equations 
have been commonly utilized to understand complex flow phenomena, playing a significant role in the developments of modern infrastructures such as cars, high-speed trains, airplanes, 
and wind turbine generators. On the other hand, owing to the recent improvements in observation and measurement technologies, data-driven approaches such as machine learning and 
data analysis are also becoming another powerful tool to extract useful information from complex flow phenomena. Whichever approach we may take, with the explosion of data size 
brought by these numerical simulations and observation/measurements, it is strongly desired to develop an efficient way to describe and extract latent knowledge hidden behind 
complex flow phenomena and to make better predictions from those big flow data. One approach to responding to the requirements is to identify topological features 
of fluid patterns in terms of local orbit structures such as singular orbits, periodic orbits, and connecting orbits between them. This is because,  in many experiments/measurements of 
complex flow phenomena and numerical simulations of flow equations, we often visualize the orbits of particles advected by the vector field to understand the flow structures.
On the other hand, since it is too complicated to extract such topological structures of three-dimensional flows in general, we restrict our attention to flow patterns in 2D spaces or 
we consider the projection of 3D flows on 2D spaces for convenience.}

One of the most important classes of 2D flows is divergence-free vector fields on a two-dimensional domain without a handle. Any divergence-free vector field is then generated 
by a smooth scalar function, called {\it Hamiltonian}. When a passive particle is advected by a $C^r$ ($r\geq 1$) divergence-free vector field,
the orbit of the passive scalar corresponds to a level curve of the Hamiltonian. 
Hence, to realize the topological classification of flow patterns, it is sufficient to pay attention to the topological structures of these level curves consisting of singular orbits and connecting orbits like Morse theory.
 By analogy in the flow of incompressible fluids on planes, the Hamiltonian is called the stream function,  and its level curves are referred to as streamlines.
 
 There are some preceding results on the topological classification of orbit structures for incompressible and irrotational flows with a finite number of point vortices~\cite{ArBr98,KiNe00,Moffatt01}. It is generalized for integrable Hamiltonian dynamical systems with one degree of freedom on closed surfaces~\cite{BoFo04} {and unbounded surfaces \cite{Nikolaenko20,Y21}}. 
Ma and Wang~\cite{MaWa05} have provided a classification of streamline topology of {\it structurally stable} Hamiltonian flows. 
 Here, by a structurally stable Hamiltonian flow, we mean a Hamiltonian flow whose topological structure remains unchanged under any infinitely small
smooth perturbations in the set of Hamiltonian flows. The theory has been extended to the structurally stable Hamiltonian flows in the presence of a uniform 
flow and solid boundaries~\cite{YS12}.  
Later, it is found that the topological structure of particle orbits is in one-to-one correspondence with labeled directed rooted trees~\cite{SY18} and symbolic expressions called
{\it partially Cyclically Ordered rooted Tree (COT) representations}~\cite{UYS18}.  In addition, a comparison of words and trees enables us to identify a marginal structurally unstable Hamiltonian vector field between two given structurally stable Hamiltonian vector fields and to determine all possible transitions among streamline topology~\cite{SY15}.
{Using the COT representations, one of the authors has recently constructed the complete graphs describing all transitions among topological patterns of structurally 
stable Hamiltonian vector fields having three and four genus elements~\cite{YY21}.}

 The discrete tree expressions and the symbolic representations of 2D Hamiltonian flow topology are not just discrete identifies of flow topology, but they are successfully utilized to reveal the relation 
 between topological orbit structures with the functionality realized by the flows.  For instance, the numerical computation of a flow around an inclined flat plate in a uniform 
 flow \cite{SSY14}  demonstrated that, whenever the lift-to-drag ratio acting on the plate attains local maxima, the symbolic representations for the instantaneous streamline 
 topology share a common sequence of letters and we find that the common sequence symbolizes the existence of an entrapped vortex structure above the plate. 
 This indicates the common letters encode the existence of a lift-enhancing vortex structure. Recently, computer software converting the Hamiltonian values on grid points into
  its COT representation has been developed.  {Applying the software to a large number of datasets in meteorology and oceanography, we successfully identify the occurrence of abnormal geophysical
  events such as atmospheric blocking events in the northern hemisphere~\cite{USIK20} and the large meandering of Kuroshio current along the south coast of Japan island~\cite{SOU22}.
  It thus brings a new methodology of data analysis -- what we call ``topological flow data analysis(TFDA)''.}

{The theory of the topological classification of Hamiltonian vector fields can be extended in various ways. Theoretically, hyperbolic vector fields are a significant class of 2D vector fields.
 Andronov-Pontryagin~\cite{andronov1937rough} and Peixoto~\cite{Peixoto} have revealed that 
the set of Morse-Smale $C^r$-vector fields ($r \geq 1$) on an orientable closed surface is open dense and structurally stable in the set of $C^r$-vector fields.
Morse-Smale flows on closed surfaces are classified and the complete invariant for them has been constructed~\cite{BN97,nikolaev2001foliations,OS,wang1990c}.
 They are then generalized into $\Omega$-stable flows on closed surfaces, which are Morse-Smale flows with heteroclinic orbits~\cite{KMP}, and Morse flows on compact 
 surfaces~\cite{Prishlyak2020}.  In~\cite{KMP2018}, a polynomial-time algorithm was provided for recognizing the equivalence and conjugacy 
 class of a Morse-Smale flow, called a gradient-like flow,  on a closed surface. }
 
{Another important class of 2D vector fields is {\it the flow of finite type} that was introduced in~\cite{Y}, whose definition is given in Section~3. This class is of practical significance since it
 contains 2D compressible vector fields and 2D flow fields obtained by projecting or restricting 3D vector fields on certain 2D surfaces.  Hence, the discrete tree representations
  and the COT representations for the flow of finite type certainly become useful tools for topological identification of flow patterns obtained by measurements and numerical simulations 
  for complicated 3D flows.

The purpose of this study is to extend topological classification theory to the flows of finite type and to provide discrete tree expressions expressed by COT representations for them.}
The paper consists of six sections. Recalling some mathematical definitions from the theory of dynamical systems in Section~2, we define the flow of finite type, and we
 then introduce its topological structure theorem in Section~3.  In Section~4, after describing all local orbit structures associated with COT symbols, we show that the topological orbit
  structures generated by the flows of finite type are in one-to-one correspondence with a COT representation as a rooted tree and linking structure as surface graphs, which are new 
discrete expressions for the flow topology. In Section~5, we provide some applications to  2D compressible flows and applications to 3D vector fields appearing in industrial problems.  
In {Section~6 and} the last section, a summary and future works are described. 

\section{Mathematical preliminaries: components in surface flows} \label{2-1}
 By a surface $S$, we mean a compact two-dimensional manifold, which is homeomorphic to a spherical surface $\mathbb{S}^2$ with open punctured disks. 
 That is to say, we allow the existence of
 several solid boundaries on the surface. An unbounded plane with a point at infinity is topologically identified with a surface $S$ by the stereographic projection.
  We here consider a continuous $\mathbb{R}$-action on 
 the surface $S$ denoted by  $v: \mathbb{R} \times S \rightarrow S$,
 called a flow on $S$. Then, for any $t \in \mathbb{R}$, we define a map $v_t:S \rightarrow S$ by
 $v_t:=v(t, \cdot )$. An orbit starting from $x \in S$ is then given by 
 \[
 O(x):=\{ v_t(x) \in S\, \vert \, t \in \mathbb{R}\}.
 \]
  For a $C^r$ ($r\geq 1$) vector field $u(x)$, the orbit is defined by the solution $\varphi(t;x)$ of the initial value problem $\dot{\varphi} = u(\varphi(t)), \varphi(0)=x \in S$
   as $v_t :=\varphi(t; \cdot)$.
 Since the solution exists uniquely owing to the smoothness of the vector field, 
 the flow on $S$ is well-defined. An orbit is said to be {\it proper} if there is a neighborhood of it where the orbit becomes a  closed set. According to~\cite{HH,Mai43}, any orbit on a 
 surface $S\subseteq \mathbb{S}^2$ is proper. To classify the proper orbits on surfaces, we introduce some special orbits in what follows.
 
 \begin{definition}
 A point $x \in S$ is a {\it singular} orbit, if $x=v_t(x)$ for any $t \in \mathbb{R}$, \rm{i.e.} $O(x)=\{x\}$.
 An orbit $O(x)$ is said to 
 be {\it periodic}, if there is a positive real $T>0$ such that $x=v_T(x)$ and $x \neq v_t(x)$ for any $t \in (0,T)$. 
 We say an orbit is {\it non-closed}, if it is neither singular nor periodic. 
  \end{definition}
Let  $\mathrm{Sing}(v)$, $\mathrm{Per}(v)$ and $\mathrm{P}(v)$ denote the sets of singular orbits,  the union of periodic orbits  and 
non-closed orbits respectively. The union of proper obits,  say $\mathrm{Pr}(v)$, is then decomposed into $\mathrm{Pr}(v) =  \mathrm{Sing}(v) \sqcup \mathrm{Per}(v) \sqcup \mathrm{P}(v)$ by definition. Moreover, let us recall a notion in the theory of dynamical systems. 
 \begin{definition}
 The $\omega$-limit set of $x \in S$ is defined by  $\omega(x):=\cap_{n \in \mathbb{R}}\overline{\{v_t(x)\, \vert \, t>n\}}$, where
 $\overline{A}$ denotes the closure of a subset $A$. Similarly, one defines the $\alpha$-limit set of $x \in S$ by 
 $\alpha(x):=\cap_{n \in \mathbb{R}}\overline{\{v_t(x)\, \vert \, t<n\}}$.
 A {\it separatrix} is an orbit whose $\alpha$-limit or $\omega$-limit set is a singular orbit.
 A separatrix $\gamma$ is said to be {\it connected} if both $\omega(\gamma)$ and $\alpha(\gamma)$ are singular orbits.
 \end{definition}
We first define point flow components in $\mathrm{Sing}(v)$. See Figure~\ref{Fig01:PFC}.
\begin{definition}(Point flow components) Let $x \in \mathrm{int}S$ be an isolated singular orbit.
\begin{itemize}
\item $x$ is called a {\it sink} (resp. a {\it source}), if there is a neighborhood $U$ of $x$ such that $\omega(y)=\{ x\}$  (resp. $\alpha(y)=\{ x\}$) for any $y \in U$,
\item $x$ is called a {\it center}, if there is a neighborhood $U$ of $x$ such that $U-\{ x\}$ is filled with periodic orbits,
\item $x$ is called a {\it saddle}, if it has exactly four separatrices. 
\end{itemize}
For a singular orbit on a solid boundary $x \in \partial S$,
\begin{itemize}
\item $x$ is called a {\it $\partial$-sink} (resp. a {\it $\partial$-source}), if there  is a neighborhood $U$ of $x$ such that $\omega(y)=\{ x\}$ (resp. $\alpha(y)=\{x\}$) for any $y \in U$, 
\item $x$ is called a {\it $\partial$-saddle} if it has exactly three separatrices.
\end{itemize}
\end{definition}

\begin{figure}
\begin{center}
\includegraphics[scale=0.45]{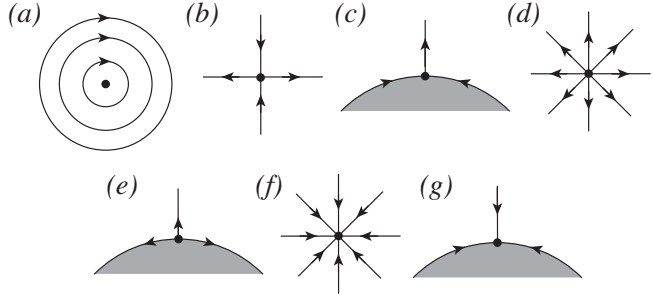}
\end{center}
\caption{Point components. (a) Center, (b) saddle, (c) $\partial$-saddle, (d) source, (e) $\partial$-source, (f) sink, (g) $\partial$-sink. }
\label{Fig01:PFC}
\end{figure}

Next, we introduce one-dimensional flow components.  
A subset is invariant if it is a union of orbits. 
For a mapping $f\colon \mathbb{S}^1 \rightarrow S$, if the image of a circle is invariant and is an isolated point, it is called a {\it trivial} circuit. On the other hand, if $f$ is an immersion and its image is invariant, then we call it a {\it non-trivial} circuit. A trivial circuit is an isolated singular orbit in $\mathrm{Sing}(v)$, while a non-trivial circuit represents either a {\it cycle}, which is a periodic orbit in $\mathrm{Per}(v)$ as in Figure~\ref{Fig02:LC}(a), or a one-dimensional 
 connected subset composed of singular orbits in $\mathrm{Sing}(v)$ and separatrices in $\mathrm{P}(v)$ between them as shown in Figure~\ref{Fig02:LC}(b).
 Now we define limit circuits as follows.
   \begin{definition}(Limit circuit)
 A {\it limit circuit} $\gamma$ is a non-trivial circuit which satisfies
 either $\alpha(x)=\gamma$ or $\omega(x)=\gamma$ for some point $x \notin \gamma$.
 \end{definition}
  For a limit circuit $\gamma$, $W^s(\gamma)$ (resp. $W^u(\gamma)$) denotes the stable (resp. the unstable) manifold of $\gamma$.   A limit circuit $\gamma$ is called {\it attracting} (resp. {\it repelling}), if there is an open annulus $\mathbb{A}$ such that one boundary of $\mathbb{A}$ is $\gamma$ and $\mathbb{A} \subseteq W^s(\gamma)$ (resp. $\mathbb{A} \subseteq W^u(\gamma)$). Then the annulus $\mathbb{A}$ is called an {\it attracting} (resp. a {\it repelling}) collar of $\gamma$. Hence, by definition, any limit circuit should have 
 at least one attracting or repelling collar. We call ($\partial$-)sinks and attracting limit circuits {\it sink structures}, while ($\partial$-)sources and repelling limit circuits are referred to as
 {\it source structures}. The following one-dimensional flow components belong to $\mathrm{P}(v)$.
  \begin{definition}(Saddle separatrix, $\partial$-saddle separatrix and ss-separatrix) A {\it saddle separatrix} is a connecting orbit whose $\alpha$-limit
 and $\omega$-limit sets are either a saddle or a $\partial$-saddle. When a saddle separatrix is connecting
 the same saddle, it is called {\it self-connected}. A saddle separatrix between
 two $\partial$-saddles on the same boundary is  referred to as a {\it self-connected $\partial$-saddle  separatrix}. 
 \end{definition}
 \begin{definition}(ss-component and ss-separatrix) An {\it ss-component} is either a ($\partial$-)sink, a ($\partial$-)source, or a limit circuit. 
 A separatrix connecting between a ($\partial$-)saddle and an ss-component is called an {\it ss-separatrix}.
 \end{definition}
  Using these flow components, we define a ``skeleton'' of orbit structures, called the {\it ss-saddle connection diagram}. 
 \begin{definition}(ss-saddle connection diagram) For a given flow $v$ on a surface, the union of saddles, $\partial$-saddles, saddle separatrices, ss-components 
 and ss-saddle separatrices is called the ss-saddle connection diagram, which is denoted by $D_{ss}(v)$.
 \end{definition}

 \begin{figure}
\begin{center}
\includegraphics[scale=0.38]{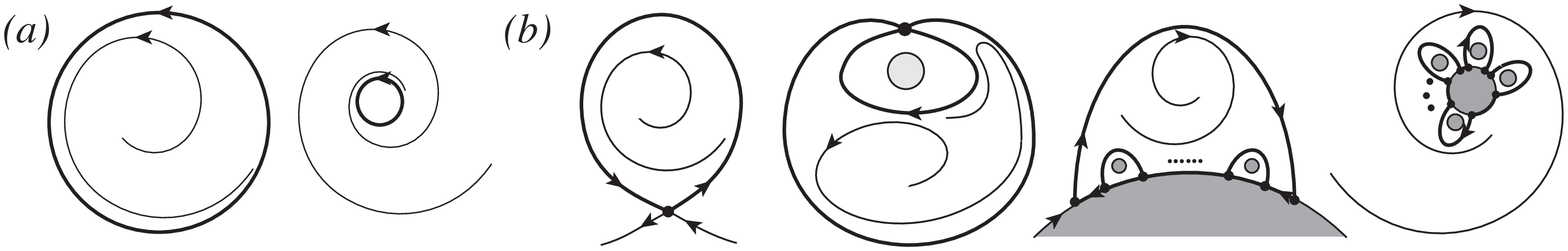}
\end{center}
\caption{Examples of non-trivial limit circuits in (a) $\mathrm{Per}(v)$ and (b) $\mathrm{P}(v)\sqcup \mathrm{Sing}(v)$. They are shown in thick black lines.}
\label{Fig02:LC}
\end{figure}

Suppose that a saddle $x$ has
 four separatrices $\gamma_1$, $\gamma_2$, $\gamma_3$, $\gamma_4$ with $\alpha(\gamma_1)=\alpha(\gamma_3)=\omega(\gamma_2)=\omega(\gamma_4)=x$ as in
 Figure~\ref{Fig05:SS}(a,b). We then introduce  special orbit structures, called slidable ($\partial$-)saddles. 
\begin{definition}(Slidable saddle) 
The saddle $x$ is called {\it slidable}, if either $\omega(\gamma_1)$ is a sink which is not $\omega(\gamma_3)$ in Figure~\ref{Fig05:SS}(a), or
 $\alpha(\gamma_2)$ is a source which is not $\alpha(\gamma_4)$ in Figure~\ref{Fig05:SS}(b). 
 \end{definition}
Similarly, let $\gamma \in \mathrm{int}S$ and $\mu \in \partial S$ denote two separatrices of a $\partial$-saddle with $\omega({\gamma})=x$ and $\alpha({\mu})=x$
as in Figure~\ref{Fig05:SS}(c). Then we have the following definition. 
\begin{definition}(Slidable $\partial$-saddle) 
A  $\partial$-saddle $x \in \partial S$ is said to be slidable, if 
there exists another 
$\partial$-saddle $y$ on the same boundary such that $\omega(\mu) =y$ as in Figure~\ref{Fig05:SS}(c) or (d).
\end{definition}  

\begin{figure}
\begin{center}
\includegraphics[scale=0.6]{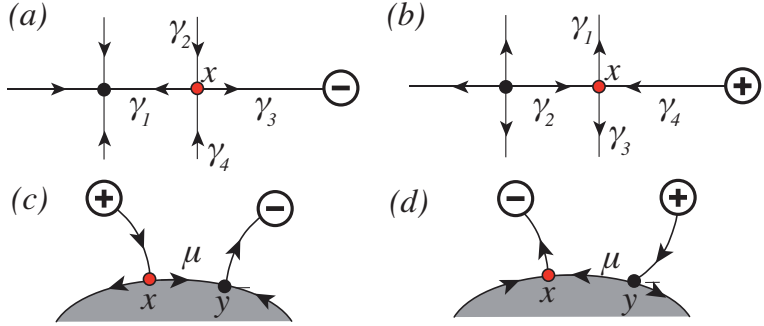}
\end{center}
\caption{(a,b) A point $x$ is a slidable saddle. (c,d) A point $x$ is a slidable $\partial$-saddle and $y$ is its associated $\partial$-saddle.
 The symbol $\oplus$ (resp. $\ominus$) indicates that the separatrix connecting to $y$ is coming from a source (resp. a sink) structure whose divergence is positive
  (resp. negative).} 
\label{Fig05:SS}
\end{figure}

\section{Flow of finite type and the topological structure theorem}
 We consider a class of 2D vector fields, which was introduced in \cite{Y}.
\begin{definition}
A flow on the surface is called {\it a flow of (strongly) finite type}, if the following are satisfied.
\begin{description}
\item[(1)] All orbits generated by the flow are proper;
\item[(2)] All singular orbits  are non-degenerate;
\item[(3)] The number of limit cycles is finite;
\item[(4)] All saddle separatrices are self-connected.
\end{description}
\end{definition}
 The first condition is satisfied because the surface is contained in a sphere as mentioned in Section~\ref{2-1}.
 The second condition shows that all singular orbits are isolated and finite. 
  The third condition is necessary to extract a finite number of orbit structures from the flows since the existence of infinitely many limit cycles implies uncountably
  many local topological structures. 
 Regarding the fourth condition, since it has been shown in~\cite{YS12} that non-self-connected separatrices of Hamiltonian vector fields are structurally unstable, we exclude the existence of such orbits so that this class becomes an extension of structurally stable Hamiltonian vector fields.  
Though a flow satisfying conditions $(1)$--$(3)$ in the definition is called a flow of finite type in \cite{Y},   we, therefore, assume the generic condition $(4)$. 
 
 Let us here recall that a $C^r$  ($r \geq 1$) vector field satisfying three generic conditions for differentials to guarantee structural stability on a surface $S$ is called {\it Morse--Smale}~\cite{LaPa90} if it satisfies the following four conditions.
(1) Each singular orbit (orbit) is hyperbolic (i.e. a ($\partial$-)sink, a ($\partial$-)source, or a ($\partial$-)saddle);
(2) Each periodic orbit is a hyperbolic limit cycle;
(3) Each saddle separatrix is contained in $\partial S$;
(4) The $\omega$-limit ($\alpha$-limit) set of a point is a closed orbit. 
Morse--Smale flows without periodic orbits are just called {\it Morse flows}.
One can then show that every Morse--Smale flow (i.e. a flow generated by a Morse--Smale vector field) is also of finite type: Since the $\omega$-limit ($\alpha$-limit) set of 
a point is a closed orbit, each orbit is proper.  Hyperbolicity of singular orbits implies that each singular orbit is non-degenerate. 
Since each saddle separatrix is contained in $\partial S$, it is self-connected and there 
are no non-periodic limit circuits. Hyperbolicity of limit cycles implies the existence of a finite number of limit circuits. 
 
 The flow of finite type is thus a generalized class of vector fields containing structurally stable Hamiltonian flows as well as Morse--Smale flows, which are 
 of mathematical significance class of vector fields~\cite{KMP,MaWa05,OS,Peixoto,YS12}.
 In addition, from the application point of view, experimental/measurement data subject to noises and errors and a 2D projection of 
3D incompressible flows admit a finite number of local orbit structures belonging to the flow of finite type.
  
   \begin{figure}
\begin{center}
\includegraphics[scale=0.6]{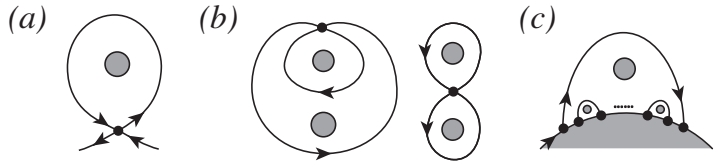}
\end{center}
\caption{Examples of non-trivial limit circuits generated by the flow of finite type. (a) A self-connected saddle separatrix. 
(b) Two self-connected saddle separatrices. (c) Self-connected $\partial$-saddle separatrices.} 
\label{Fig08:NLC}
\end{figure}

  Since a saddle has four separatrices, non-trivial limit circuits of the flow of finite type must be either a saddle with
   one self-connected saddle separatrix as Figure~\ref{Fig08:NLC}(a), or a saddle with two self-connected saddle separatrices
    as Figure~\ref{Fig08:NLC}(b). Then the limit circuit is given by the union of saddles and their connecting self-connected separatrices. As for
    $\partial$-saddles with three separatrices, a non-trivial circuit is defined by the union of one self-connected $\partial$-saddle separatrix on a boundary
    and its enclosing self-connected $\partial$-saddle separatrices attached to the same boundary.
    See Figure~\ref{Fig08:NLC}(c). Moreover, we have the following lemma regarding the $\omega$-limit and $\alpha$-limit sets.
    \begin{lemma}\label{lem51} 
The $\omega$-limit set $($resp. the $\alpha$-limit set$)$ of a non-closed orbit is one of the following: (1) a saddle, (2) a $\partial$-saddle, (3) a sink, (4) a source,
  (5) a $\partial$-sink, (6) a $\partial$-source, (7) an attracting $($resp. a repelling$)$ limit cycle (8) an attracting $($resp. a repelling$)$ non-trivial limit circuit.
\end{lemma}

\begin{proof}
By a generalized Poincar\'e--Bendixson theorem (Theorem 2.6.1 \cite{NZ}), 
the $\omega$-limit  (resp. the $\alpha$-limit) set of a non-closed orbit is either 
a singular orbit, an attracting (resp. a repelling) limit circuit, or a quasi-minimal set.
Since $S$ is filled with proper orbits, there are no quasi-minimal sets.
Since $v$ is regular, each singular orbit is either a center, a saddle, a $\partial$-saddle, 
a $\partial$-sink, a $\partial$-source, a sink, or a source. 
\end{proof}

We assume the non-existence of $\partial$-sources and $\partial$-sinks in this paper since these structures are generically approximated by slidable $\partial$-saddles.
Then, this lemma implies that the set of orbits generated by the flows of finite type is decomposed
into three categories; (i) Limit sets and their connecting non-closed orbits belonging to
$\mathrm{Sing}(v) \sqcup \mathrm{P}(v)$, (ii) centers in $\mathrm{Sing}(v)$, (iii) non-closed orbits in $\mathrm{int}P(v)$, and (iv) periodic orbits in  the interior of the surface  belonging to $\mathrm{Per}(v)$ and those on the boundary of the surface in  $\partial S \cap \mathrm{Per}(v)$. Based on the observations, we introduce the union of orbits contained in 
$D_{ss}(v)$, called {\it border orbits}.
\begin{definition}
The union of border orbits, say $\mathrm{Bd}(v)$, for the flow of finite type $v$ on the surface $S$ is given by
\[
\mathrm{Bd}(v):= \mathrm{Sing}(v)  \cup \mathrm{P}_\mathrm{sep}(v) \cup \partial_\mathrm{Per}(v) \cup \partial\mathrm{P}(v)\cup \partial\mathrm{Per}(v),
\]
where 
\begin{itemize}
\item $\mathrm{P}_\mathrm{sep}(v)$ denotes the union of saddle separatrices and ss-separatrices in $\mathrm{int}\mathrm{P}(v)$,
\item $\partial_\mathrm{per}(v)$ denotes  the union of periodic orbits along boundaries in $\partial S \cap \mathrm{int}\mathop{\mathrm{Per}}(v)$,
\item $\partial\mathrm{P}(v)$ denotes the boundary of the union of non-closed proper orbits,
\item $\partial\mathrm{Per}(v)$ denotes the boundary of the union of periodic orbits.
\end{itemize}
\label{def:border-orbits}
\end{definition}
 Since any orbit in $\mathrm{P}_\mathrm{sep}(v)$ belongs to $\mathrm{int} \mathrm{P}(v)$, it cannot be a limit set. A periodic orbit in $\partial_{\mathrm{per}}(v)$ is
 going along a solid boundary. 
 Orbits in $\partial \mathrm{P}(v)$ and  $\partial \mathrm{Per}(v)$  consist of limit cycles and non-trivial limit circuits respectively. Any limit circuit in $\partial \mathrm{Per}(v)$ is a boundary of  $\mathrm{Per}(v)$. At the same time, it should be a boundary of $\mathrm{P}(v)$. Otherwise, it is an interior point of $\mathrm{Per}(v)$. 
  
   In the topological classification theory of structurally stable Hamiltonian flows, a unique  tree structure has been assigned based on the composition
   of two-dimensional domains separated by the border orbits~\cite{UYS18}.
  In the same spirit, we classify two-dimensional domains in $(\mathrm{Bd}(v))^c=S-\mathrm{Bd}(v)$ for the flow of finite type $v$ by introducing the notion of {\it orbit space}
  as follows.
\begin{definition}
Let $T \subset S$ denote a union of proper orbits generated by the flow $v$ on the surface $S$. Then the {\it orbit space} is defined as
 a quotient set $T/\sim$, where the equivalence relation $\sim$ is defined by $x\sim y$ for 
 $x,y \in T$, if $O(x)=O(y)$.
 \label{def:orbit_space}
\end{definition}
 Figure~\ref{fig:flowbox}(a) shows an open box filled with non-closed orbits such as uniform flows. It is called a {\it trivial flow box} whose 
 orbit space, by definition, is equivalent to an open interval. An open annulus filled with non-closed orbits in Figure~\ref{fig:flowbox}(b) is called 
 a {\it transverse annulus}, while an open annulus filled
  with periodic orbits in Figure~\ref{fig:flowbox}(c) is called a {\it periodic annulus}. Their orbit spaces are a circle and an open interval, respectively. Then we have the
  following theorem, called the topological structure theorem for flows of finite type.

 \begin{theorem}\label{thm:flowbox}
For any flow of finite type on a surface $S  \subseteq \mathbb{S}^2$, 
each connected component of $(\mathop{\mathrm{Bd}(v)})^c=S - \mathop{\mathrm{Bd}(v)}$ is either  a trivial flow box in $\mathrm{P}(v)$, a transverse 
annulus in $\mathrm{P}(v)$, or a periodic annulus in $\mathop{\mathrm{Per}}(v)$.
\end{theorem}
 This theorem claims that there are three types of two-dimensional domains separated by
the border orbits shown in  Figure~\ref{fig:flowbox}.  The proof has been provided in~\cite{Y} in a general mathematical framework.
\begin{figure}
\begin{center}
\includegraphics[scale=0.4]{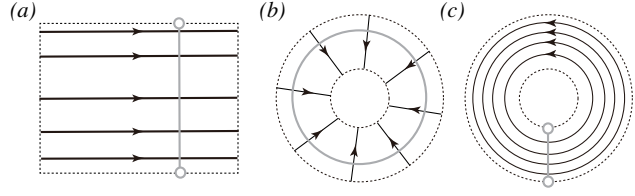}
\end{center}
\caption{Two-dimensional domains in $(\mathrm{Bd}(v))^c$ for the flow of finite type $v$, and their corresponding orbit spaces. Orbits filling the domains are drawn as black
solid lines with arrows, while their orbit spaces are shown as dashed gray curves. (a) A trivial flow box whose orbit space is an open interval. (b) A transverse annulus whose orbit space is a circle. (c) A periodic annulus whose orbit space is an open interval.}
\label{fig:flowbox}
\end{figure}

\section{Combinatorial representations for orbit structures}
\subsection{Local orbit structures in $\mathrm{Bd}(v)$ and $(\mathrm{Bd}(v))^c$ and COT symbols}
 We introduce all local orbit structures contained in $\mathrm{Bd}(v)$ and $(\mathrm{Bd}(v))^c$ in Table~\ref{tbl:All_structures}, which are categorized into some classes as summarized in
 Table~\ref{tbl:COT_Structures}. At the same time, we provide a unique symbol, called a {\it COT symbol}, associated with each local orbit structure. 
 The COT symbols express local orbit structures and the composition of their surrounding orbit structures.

\begin{table}
\begin{center}
\caption{All orbit structures generated by the flow of finite type and their COT symbols. Double-signs correspond.}
\begingroup
\scalefont{0.93}
\begin{tabular}{|l|c|l|l|}\hline
\multicolumn{2}{|c|}{Root structure} & COT symbol & Figure\\ \hline
center & $\sigma_{\emptyset\pm}$ &$\sigma_{\emptyset\mp}(\B_{b\emptyset\pm})$  & \ref{fig:Fundamental_structures}(a) \\ \hline
& $\sigma_{\emptyset \widetilde{\pm}0}$ & $\sigma_{\emptyset \widetilde{\mp}0}(b_{\widetilde{\pm}}(\B_{\widetilde{\pm}}, \{ \B_{a\widetilde{\pm}s} \}))$ & \ref{fig:Fundamental_structures}(a)  \\  \cline{2-4}
sink/source & $\sigma_{\emptyset \widetilde{\pm}+}$ & $\sigma_{\emptyset \widetilde{\mp}+}(b_{\widetilde{\pm}}(\B_{\widetilde{\pm}}, \{ \B_{a\widetilde{\pm}s} \}))$  & \ref{fig:Fundamental_structures}(a) \\ \cline{2-4}
& $\sigma_{\emptyset \widetilde{\pm}-}$ & $\sigma_{\emptyset \widetilde{\mp}-}(b_{\widetilde{\pm}}(\B_{\widetilde{\pm}}, \{ \B_{a\widetilde{\pm}s} \}))$  & \ref{fig:Fundamental_structures}(a) \\ \hline
\multirow{2}{*}{boundary} & $\beta_{\emptyset\pm}$ & $\beta_{\emptyset\mp}(\B_{b_\pm}, \{ \B_{c\mp s} \})$& \ref{fig:Fundamental_structures}(b)\\ \cline{2-4}
& $\beta_{\emptyset 2}$ & $\beta_{\emptyset 2}(\{ \B_{c+s}, \B_{\widetilde{+}}, \B_{c-s}, \B_{\widetilde{-}}, \B_{\gamma + s}\}, \B_{as}  )$ &\ref{fig:Fundamental_structures}(c) \\\hline\hline
\multicolumn{2}{|c|}{2D structure} & COT symbol & Figure\\ \hline
transverse annulus & $b_{\widetilde{\pm}}$ & $b_{\widetilde{\pm}}(\B_{\widetilde{\pm}}, \{ \B_{a\widetilde{\pm}s} \})$&\ref{fig:Saturated_structures}(a)  \\ \hline
periodic annulus &  $b_\pm$ & $b_\pm(\B_{\alpha\pm})$ & \ref{fig:Saturated_structures}(b) \\ \hline\hline
\multicolumn{2}{|c|}{Isolated structure} & COT symbol & Figure \\ \hline
center  & $\sigma_\pm$ & $\sigma_\pm$ & \ref{fig:Boundary_isolated_point}(b) \\ \hline
 &   $\sigma_{\widetilde{\pm}0}$ & $\sigma_{\widetilde{\pm}0}$& \ref{fig:Boundary_isolated_point}(b)  \\ \cline{2-4}
sink/source &  $\sigma_{\widetilde{\pm}\mp}$ & $\sigma_{\widetilde{\pm}\mp}$ & \ref{fig:Boundary_isolated_point}(b)  \\ \cline{2-4}
 &   $\sigma_{\widetilde{\pm}\pm}$ & $\sigma_{\widetilde{\pm}\pm}$& \ref{fig:Boundary_isolated_point}(b)  \\ \hline
\multicolumn{2}{|c|}{Cycles} & COT symbol & Figure \\ \hline
\multirow{2}{*}{periodic orbit} &  $p_{\widetilde{\pm}}$ & $p_{\widetilde{\pm}}(\B_{b\widetilde{\pm}})$ & \ref{fig:Boundary_isolated_point}(c)  \\ \cline{2-4}
&  $p_\pm$ & $p_\pm(\B_{b\pm})$ &  \ref{fig:Boundary_isolated_point}(d)   \\ \hline\hline
\multicolumn{2}{|c|}{Circuits} & COT symbol & Figure\\ \hline
ss-saddle connection&  $a_\pm$ & $a_\pm(\B_{b\pm})$& \ref{fig:S1_structures}(a)  \\ \cline{2-4}
for a saddle &  $q_\pm$ & $q_\pm(\B_{\widetilde{+}}, \B_{\widetilde{-}}, \B_{as})$ & \ref{fig:S1_structures}(b) \\ \hline
\multirow{2}{*}{figure eight} &  $b_{\pm\pm}$ & $b_{\pm\pm} \{ \B_{b\pm}, \B_{b\pm} \}$& \ref{fig:S2_structures}(a)  \\ \cline{2-4}
&  $b_{\pm\mp}$ & $b_{\pm\mp}(\B_{b\pm}, \B_{b\mp})$& \ref{fig:S2_structures}(b)   \\ \hline
periodic $\partial$ component &  $\beta_\pm$ & $\beta_\pm\{ \B_{c\pm s} \}$& \ref{fig:S4_structures}(a) \\ \hline
self-connected  &  $c_\pm$ & $c_\pm(\B_{b\pm}, \B_{c\mp s})$& \ref{fig:S4_structures}(b) \\ \cline{2-4}
separatrices on $\partial$ &  $c_{2\pm}$ & $c_{2\pm}(\B_{c\mp s}, \B_{\widetilde{\pm}}, \B_{c\pm s}, \B_{\widetilde{\mp}}, \B_{\gamma \pm s},\B_{c\mp s},  \B_{as})$& \ref{fig:S4_structures}(c)  \\  \hline\hline
\multicolumn{2}{|c|}{Slidable saddle} & COT symbol & Figure \\ \hline
&  $a_2$ & $a_2(\B_{c+s}, \B_{c-s}, \B_{\gamma-s})$ & \ref{fig:S5_structures}(a)  \\  \cline{2-4}
& $\gamma_{\widetilde{+}-}$ & $\gamma_{\widetilde{+}-}(\B_{c-s}, \B_{\widetilde{+}}, \B_{c+s})$&\ref{fig:S5_structures}(b) \\ \cline{2-4}
slidable $\partial$-saddle& $\gamma_{\widetilde{-}-}$ & $\gamma_{\widetilde{-}-}(\B_{c-s}, \B_{c+s}, \B_{\widetilde{-}})$ &\ref{fig:S5_structures}(b)\\ \cline{2-4}
& $\gamma_{\widetilde{+}+}$ & $\gamma_{\widetilde{+}+}(\B_{c+s}, \B_{c-s}, \B_{\widetilde{+}})$ &\ref{fig:S5_structures}(c) \\  \cline{2-4}
& $\gamma_{\widetilde{-}+}$ & $\gamma_{\widetilde{-}+}(\B_{c+s}, \B_{\widetilde{-}}, \B_{c-s})$ &\ref{fig:S5_structures}(c) \\  \hline
\multirow{2}{*}{slidable saddle} &  $a_{\widetilde{\pm}}$ & $a_{\widetilde{\pm}}\{ \B_{\widetilde{\pm}},  \B_{\widetilde{\pm}}\}$&\ref{fig:S3_structures}(a)  \\ \cline{2-4}
&   $q_{\widetilde{\pm}}$ & $q_{\widetilde{\pm}}( \B_{\widetilde{\pm}})$&\ref{fig:S3_structures}(b)   \\ \hline
\end{tabular}
\endgroup
  \label{tbl:All_structures}
\end{center}
\end{table}

 \begin{table}
 \caption{Classes of orbit structures used in COT symbols.}
\begin{center}
\begingroup
\scalefont{0.8}
\begin{tabular}{|l|c|c|c|}\hline
Class & Boxes  & Orbit structures & Note \\ \hline
class-$b_{\emptyset \pm}$ &$\B_{b_{\emptyset \pm}}$ &  $\{  b_\pm \}$ & \multirow{4}{*}{2D orbit structures} \\ \cline{1-3}class-$b_+$ &$\B_{b_+}$ &  $\{  b_{\widetilde{\pm}}, b_+ \}$ &  \\ \cline{1-3}
class-$b_-$ &$\B_{b_-}$ &  $\{  b_{\widetilde{\pm}}, b_- \}$ &  \\ \cline{1-3}
class-$b_{\widetilde{\pm}}$ &$\B_{b_{\widetilde{\pm}}}$ & $\{ b_{\widetilde{\pm}} \}$ & \\ \hline
class-$\widetilde{+}$ & $\B_{\widetilde{+}}$ & $\{ p_\pm, b_{\pm\pm}, b_{\pm\mp}, q_{\pm}, \beta_\pm, a_{\widetilde{+}}, \sigma_{\widetilde{+}\pm}, \sigma_{\widetilde{+}0}\}$ & source structures\\ \hline
class-$\widetilde{-}$ & $\B_{\widetilde{-}}$ & $\{ p_\pm, b_{\pm\pm}, b_{\pm\mp}, q_{\pm}, \beta_\pm, a_{\widetilde{-}},\sigma_{\widetilde{-}\pm}, \sigma_{\widetilde{-}0} \}$ & sink structures \\ \hline
class-$\alpha_+$ & $\B_{\alpha_+}$ & $\{ p_{\widetilde{+}}, b_{++}, b_{+-}, q_+, \beta_+, \sigma_+ \}$ & used in $b_{+}$\\ \hline
class-$\alpha_-$ & $\B_{\alpha_-}$ & $\{ p_{\widetilde{-}}, b_{-+}, b_{--}, q_-, \beta_-, \sigma_- \}$ & used in $b_{-}$\\ \hline
class-$a$ &$\B_a$ & $\{ a_\pm, a_2\}$ & between sink/source structures \\ \hline
class-$a_{\widetilde{+}}$ &$\B_{a\widetilde{+}}$ & $\{ q_{\widetilde{-}}, a_\pm, a_2\}$ & used in $b_{\widetilde{+}}$\\ \hline
class-$a_{\widetilde{-}}$ &$\B_{a\widetilde{-}}$ & $\{ q_{\widetilde{+}}, a_\pm, a_2\}$ & used in $b_{\widetilde{-}}$ \\ \hline
class-$c_+$ &$\B_{c_+}$ &$\{ c_+, c_{2+}\}$ & \\ \cline{1-3}
class-$c_-$ &$\B_{c_-}$ &$\{ c_-, c_{2-}\}$ & \multirow{2}{*}{used near boundaries}  \\ \cline{1-3}
class-$\gamma_+$ &$\B_{\gamma_+}$ & $\{\gamma_{\widetilde{\pm}+} \}$ &  \\ \cline{1-3}
class-$\gamma_-$ &$\B_{\gamma_-}$ & $\{\gamma_{\widetilde{\pm}-} \}$ &  \\ \hline
\end{tabular}
\endgroup
\label{tbl:COT_Structures}
\end{center}
\end{table}

\subsubsection{Two-dimensional orbit structures in $(\mathrm{Bd}(v))^c$}
 We  assign no symbol to the trivial flow box in Figure~\ref{fig:flowbox}(a), since we regard it as the default orbit structure.  
 For the domains in Figure~\ref{fig:flowbox}(b,c), we assign the following COT symbols. They are categorized into class-$b_{\widetilde{\pm}}$ and class-$b_\pm$. See Table~\ref{tbl:COT_Structures}.

\begin{figure}
\begin{center}
\includegraphics[scale=0.4]{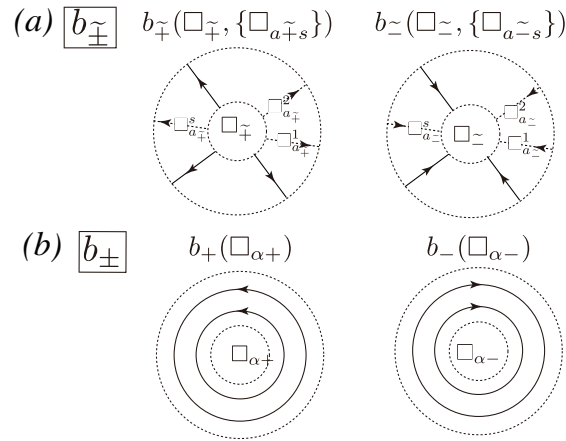}
\end{center}
\caption{The COT symbols for the orbit structures in $(\mathrm{Bd}(v))^c$. (a) Transverse annuli, whose COT 
symbols are  $b_{\widetilde{\pm}}(\B_{\widetilde{\pm}},  \{\B_{a\widetilde{\pm}s} \})$. (b) Periodic annuli, whose COT symbols are given by  $b_\pm(\B_{\alpha\pm})$. They are in double-sign correspondence.}
\label{fig:Saturated_structures}
\end{figure}

{\bf [Orbit structures $b_{\widetilde{\pm}}$]}  The COT symbol of an attracting transverse annulus in Figure~\ref{fig:flowbox}(b) is given by $b_{\widetilde{-}}(\B_{\widetilde{-}}, \{\B_{a\widetilde{-}s}\})$.  See also Figure~\ref{fig:Saturated_structures}(a). The subscript $\widetilde{-}$ means that the transverse annulus contains attracting non-closed orbits, 
 and the square symbol $\B_{\widetilde{-}}$ indicates that  a sink structure of class-$\widetilde{-}$, as listed in Table~\ref{tbl:COT_Structures}, is  embedded inside the inner boundary of the annulus.
 The symbol $\B_{a\widetilde{\pm}s}$ indicates that any number ($s\geq 0$) of class-$a_{\widetilde{\pm}}$ orbit structures are contained in the transverse 
 annulus. Precisely, $\B_{a\widetilde{\pm}s}$ in the COT symbol is described by either of
\begin{equation}
\B_{a\widetilde{\pm}s} := \B_{a\widetilde{\pm}}^1 \cdot \cdots \cdot  \B_{a\widetilde{\pm}}^s  \quad (s>0), \qquad  \B_{a\widetilde{\pm}s}:=\lambda_\sim \quad (s=0),
\label{COT_ATs}
\end{equation}
 in which $\lambda_\sim$ indicates that there contain no class-$a_{\widetilde{\pm}}$ orbit structures. The class-$a_{\widetilde{\pm}}$ orbit structures are
 arranged in cyclic order.
 That is to say, picking up one of them as the first structure and reading the others in the counter-clockwise direction, we assign the symbol $\B_{a\widetilde{\pm}}^i$ to the $i$th class-$a_{\widetilde{\pm}}$ orbit structure. They are then arranged sequentially with the separator ``$\cdot$'' as (\ref{COT_ATs}). To express the cyclic arrangement of (\ref{COT_ATs}) explicitly,
  they are enclosed in the parentheses $\{ \}$ in the COT symbol. If we reverse the direction of non-closed orbits, a repelling transverse annuls is represented by 
  $b_{\widetilde{+}}(\B_{\widetilde{+}},\{\B_{a\widetilde{+}s}\} )$, in which $\B_{\widetilde{+}}$ and $\B_{a\widetilde{+}s}$ denote a class-$\widetilde{+}$ source structure 
  inside the inner boundary and class-$a_{\widetilde{+}}$ orbit structures in the annulus respectively.

{\bf[Orbit structures $b_{\pm}$]}  The COT symbol of  a counter-clockwise periodic annulus  in Figure~\ref{fig:flowbox}(c) is given by  $b_+(\B_{\alpha_+})$, where $\B_{\alpha_+}$ denotes  a class-$\alpha_+$  orbit structure inside the inner boundary.  For a clockwise periodic annulus, the COT symbol becomes 
$b_-(\B_{\alpha -})$ with a class-$\alpha_-$ orbit structure  inside. See Figure~\ref{fig:Saturated_structures}(b).

\subsubsection{Border orbits in $\partial_\mathrm{per}(v)$ and $\mathrm{Sing}(v)$: $\beta_\pm$, $\sigma_{\widetilde{\pm}\pm}$, $\sigma_{\widetilde{\pm}0}$, $\sigma_\pm$} 
A periodic orbit in $\partial_\mathrm{per}(v)$, by definition, is going along a solid 
boundary with no $\partial$-saddle separatrices as in Figure~\ref{fig:Boundary_isolated_point}(a). We assign the COT symbol $\beta_+\{ \lambda_+\}$ (resp. $\beta_-\{ \lambda_- \}$)
 to a counter-clockwise (resp. a clockwise)  orbit.  Here, $\lambda_\pm$  expresses that ``no orbit structures'' are attached to the boundary. 
 Centers and sources/sinks are elements of  $\mathrm{Sing}(v)$. A center having
 counter-clockwise (resp. clockwise) periodic orbits in its neighborhood is denoted by $\sigma_{+}$ (resp. $\sigma_{-}$).
Sources/sinks are classified by the flow direction around them. This classification is unnecessary within topological equivalence but it is useful from the application points of view.
 The COT symbol of a source  associated with counter-clockwise (resp. clockwise) non-closed orbits  is given by $\sigma_{\widetilde{+}+}$ 
 (resp. $\sigma_{\widetilde{+}-}$). When the flow direction is reversed, we have a sink surrounded by clockwise (resp. counter-clockwise) non-closed orbits, 
 whose COT symbol becomes  $\sigma_{\widetilde{-}-}$ (resp. $\sigma_{\widetilde{-}+}$). On the other hand, a source (resp. a sink) associated 
 with non-rotating non-closed orbits is represented by $\sigma_{\widetilde{+}0}$ (resp. $\sigma_{\widetilde{-}0}$). See Figure~\ref{fig:Boundary_isolated_point}(b). 
 They  are elements of class-$\widetilde{\pm}$ 
 and class-$\alpha_\pm$. The border orbits $\beta_\pm$, $\sigma_\pm$ and $\sigma_{\widetilde{\pm}}$ have no internal structure. 
\begin{figure}
\begin{center}
\includegraphics[scale=0.4]{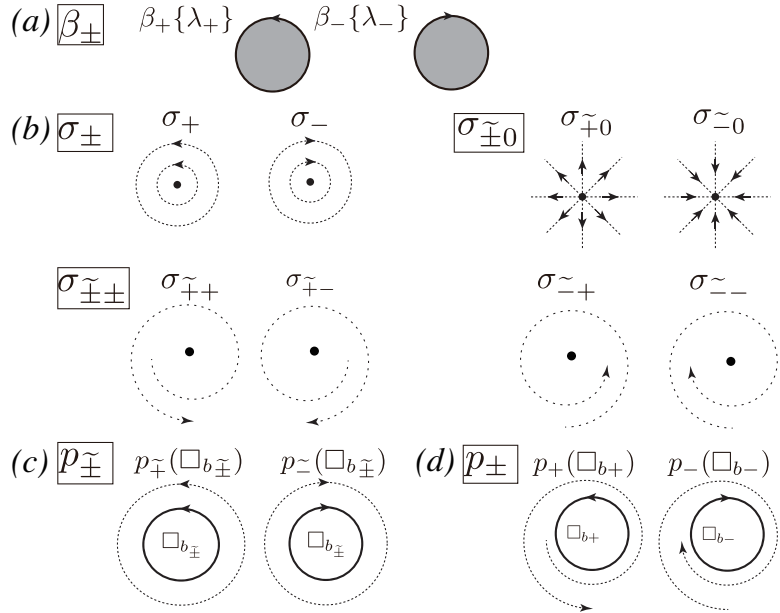}
\end{center}
\caption{(a) Border orbits in $\partial_\mathrm{per}(v)$ representing a periodic orbit along a solid boundary whose COT symbol is $\beta_\pm\{ \lambda_\pm \}$. (b) 
 Point structures in $\mathrm{Sing}(v)$. A center surrounded by counter-clockwise (resp. clockwise) periodic orbits is denoted by $\sigma_+$ (resp. $\sigma_-$). 
 A source and a sink having counter-clockwise/clockwise/non-rotating non-closed orbits in its neighborhood are represented by 
$\sigma_{\widetilde{\pm}+}$, $\sigma_{\widetilde{\pm}-}$  and $\sigma_{\widetilde{\pm}0}$, respectively. (c)  A limit cycle surrounded by periodic orbits outside, in which
 it contains an attracting/repelling transverse annulus. Its COT symbol is given by $p_{\widetilde{+}}(\B_{b\widetilde{\pm}})$ (resp. $p_{\widetilde{-}}(\B_{b\widetilde{\pm}})$) 
when the flow direction of the periodic orbits is counter-clockwise (resp. clockwise). (d) A limit cycle
with an attracting/repelling collar outside. Its COT symbol is 
$p_+(\B_{b+})$ (resp. $p_-(\B_{b-})$) when the limit cycle is going in the counter-clockwise (resp. clockwise) direction. These limit cycles are elements 
of $\partial\mathrm{P}(v)$ and $\partial\mathrm{Per}(v)$.} 
\label{fig:Boundary_isolated_point}
\end{figure}

\subsubsection{Root structures}
Let us introduce fundamental flow structures, called {\it root structures}, with their COT symbols in order to describe the conversion algorithm afterward.
\begin{figure}
\begin{center}
\includegraphics[scale=0.4]{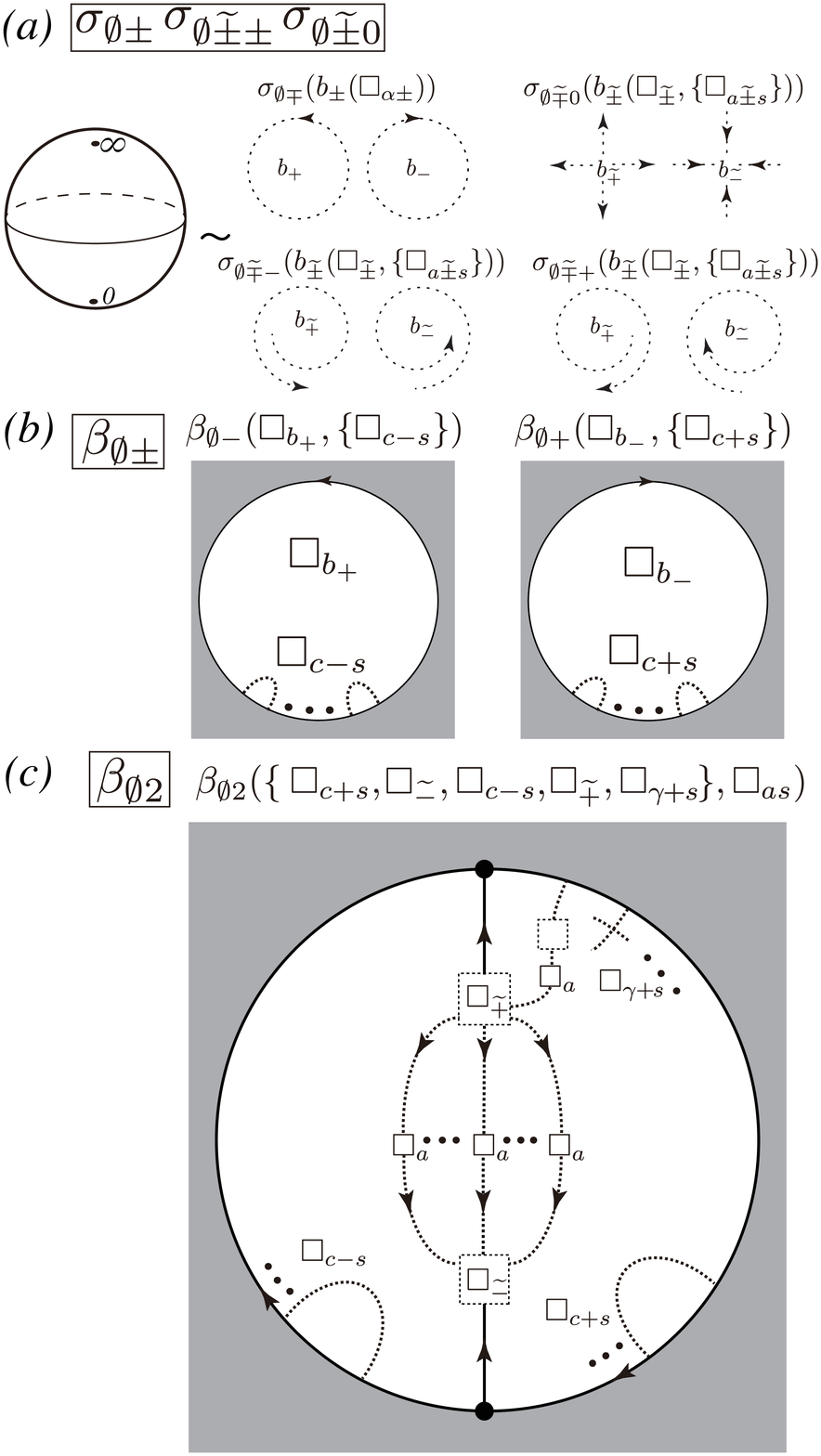}
\end{center}
\caption{Root structures for the flows of finite type on surfaces. (a) An unbounded plane without solid boundary, whose COT symbol is either 
$\sigma_{\emptyset\mp}(b_\pm(\B_{\alpha\pm}))$, $\sigma_{\emptyset\widetilde{\mp}0}(b_{\widetilde{\pm}}(\B_{\widetilde{\pm}}, \{\B_{a\widetilde{\pm}s}\}))$, $\sigma_{\emptyset\widetilde{\mp}-}(b_{\widetilde{\pm}}(\B_{\widetilde{\pm}}, \{\B_{a\widetilde{\pm}s}\}))$, 
or  $\sigma_{\emptyset\widetilde{\mp}+}(b_{\widetilde{\pm}}(\B_{\widetilde{\pm}}, \{\B_{a\widetilde{\pm}s}\}))$, depending on 
2D orbit structures  embedded in the domain. Double signs correspond. (b) Root structures $\beta_{\emptyset\pm}$, representing flows without any source/sink 
structures attached to the outer boundary. Its COT symbol is given by $\beta_{\emptyset\mp}(\B_{b_\pm}, \{ \B_{c\mp s} \})$. (c) Root structure $\beta_{\emptyset 2}$, having 
 source-sink structures, whose COT symbol is given by
$\beta_{\emptyset 2}(\{ \B_{c+s}, \B_{\widetilde{-}}, \B_{c-s}, \B_{\widetilde{+}}, \B_{\gamma + s}\}, \B_{as} )$.}
\label{fig:Fundamental_structures}
\end{figure}

{\bf (Root structures $\sigma_{\emptyset\pm}$, $\sigma_{\emptyset\widetilde{\pm}0}$ and $\sigma_{\emptyset \widetilde{\pm}\pm}$; Figure~\ref{fig:Fundamental_structures}(a))} 
 When there exist no solid boundaries in the unbounded plane,  the flow should have at least two singular orbits on the sphere which are elliptic centers, sources, and sinks. 
Moreover, suppose that there is no limit cycle and there are exactly two singular orbits. 
Then singular orbits consist of either two centers or a pair of a sink and a source.  
Without loss of generality, we may assume that the singular orbits are located at the origin and the point at the infinity of the plane. By the stereographic projection, they are identified with the north/south
 poles on the sphere $\mathbb{S}^2$.  By removing these two singular orbits, the flow of finite type on $\mathbb{S}^2$ is topologically equivalent to a periodic annulus or a transverse annulus as in Figure~\ref{fig:Fundamental_structures}(a). When the domain is a counter-clockwise periodic annulus $b_+(\B_{\alpha_+})$ around the center at the origin, 
  the flow is considered to rotate in the clockwise direction around the point at infinity. Hence, we give the COT symbol
 $\sigma_{\emptyset -}(b_+(\B_{\alpha+}))$ to the root structure. 
 
 When the flow direction is reversed, we have the COT symbol
 $\sigma_{\emptyset +}(b_{-}(\B_{\alpha-}))$.  When the singular orbit at the origin is a source (resp. a sink) associated with non-rotating  non-closed 
orbits, the orbit structure in its neighborhood is a repelling (resp. an attracting) transverse annulus represented by
 $b_{\widetilde{+}}(\B_{\widetilde{+}}, \{ \B_{a\widetilde{+}s} \} )$ (resp. $b_{\widetilde{-}}(\B_{\widetilde{-}}, \{ \B_{a\widetilde{-}s} \} )$, and
 the point at infinity becomes a sink (resp. a source). Hence, its COT symbol is given by
  $\sigma_{\emptyset\widetilde{-}0}(b_{\widetilde{+}}(\B_{\widetilde{+}}, \{ \B_{a\widetilde{+}s} \} ))$ 
  (resp. $\sigma_{\emptyset\widetilde{+}0}(b_{\widetilde{-}}(\B_{\widetilde{-}}, \{ \B_{a\widetilde{-}s} \} ))$).

When the singular orbit at the origin is a source (resp. a sink) associated with non-closed orbits rotating in the counter-clockwise direction, the point at infinity becomes a sink 
(resp. a source) associated with clockwise non-closed orbits. Then, the symbol 
$\sigma_{\emptyset\widetilde{-}-}(b_{\widetilde{+}}(\B_{\widetilde{+}}, \{ \B_{a\widetilde{+}s} \} ))$ (resp. $\sigma_{\emptyset\widetilde{+}-}(b_{\widetilde{-}}(\B_{\widetilde{-}}, \{ \B_{a\widetilde{-}s} \} ))$) is assigned to the root structure.
Reversing the flow direction, we obtain a source (resp. a sink) at infinity associated with counter-clockwise non-closed orbits, whose COT symbol becomes 
$\sigma_{\emptyset\widetilde{+}+}(b_{\widetilde{-}}(\B_{\widetilde{-}}, \{ \B_{a\widetilde{-}s} \} ))$ (resp. $\sigma_{\emptyset\widetilde{-}+}(b_{\widetilde{+}}(\B_{\widetilde{+}}, \{ \B_{a\widetilde{+}s} \} ))$).

{\bf (Root structures $\beta_{\emptyset\pm}$ and $\beta_{\emptyset 2}$; Figure~\ref{fig:Fundamental_structures}(b,c))} 
Suppose that a surface contains several solid boundaries. Then, by choosing one of them, one introduces the spherical coordinates whose 
 corresponding stereographic projection maps the flow domain on the sphere to that on a bounded domain enclosed by the chosen solid boundary from outside. 
 Figure~\ref{fig:Fundamental_structures}(b) shows an orbit structure without source/sink structures. Its COT symbol is given 
 by $\beta_{\emptyset-}(\B_{b+}, \{\B_{c-s} \} )$ when the flow direction along the outer boundary  is counter-clockwise, which is equivalent to a clockwise flow around the point at
 infinity. This structure necessarily embeds 
 a class-$b_+$ orbit structure in $\B_{b+}$, while any number ($s\geq 0$) of class-$c_-$ orbit structures can be attached to the boundary. When the flow 
 direction is reversed, we have an orbit structure represented by $\beta_{\emptyset+}(\B_{b-},  \{ \B_{c+s} \})$.
 Note that $\B_{c\pm s}$ describes the cyclic arrangement (\ref{COT_Cs}) of class-$c_\pm$ orbits structures.
 
 The root structure $\beta_{\emptyset 2}$ has sources-sink structures attached to the outer boundary as shown in Figure~\ref{fig:Fundamental_structures}(c). We pick 
 the leftmost source and sink structures as a specific pair and the other source/sink structures are  regarded as class-$\gamma_+$ orbit structures 
  in Figure~\ref{fig:S5_structures}(b).
 Along the outer boundary, it can also contain any number of class-$c_+$ (resp. class-$c_-$) orbit structures on the right (resp. the left) of the source-sink structure pair, which are 
 arranged as in (\ref{COT_Cs}). In the COT symbol, the specific source-sink pair is symbolized as $\B_{\widetilde{\pm}}$, and any number 
  $(s \geq 0)$ of the class-$\gamma_+$ orbit structures, symbolized as $\B_{\gamma + s}$, are attached to the righthand side of the outer boundary. They are arranged as
\[
\B_{\gamma + s}:= \B_{\gamma +}^1\cdot \cdots \cdot \B_{\gamma +}^s \quad (s>0), \qquad \B_{\gamma + s}:=\lambda_\sim \quad (s=0).
\]
 Note that all these orbit structures are arranged along the outer boundary in cyclic order. In addition, owing to the existence of the source-sink structure pair, there could exist any
  number of class-$a$ orbit structures, symbolized as $\B_{as}$, connecting the source structure $\B_{\widetilde{+}}$ and the sink structure $\B_{\widetilde{-}}$.  Hence, the COT symbol of 
  is provided by $\beta_{\emptyset 2}(\{ \B_{c+s}, \B_{\widetilde{-}}, \B_{c-s}, \B_{\widetilde{+}},\B_{\gamma + s}\}, \B_{as} )$.
 
\subsubsection{Limit cycles in $\partial\mathrm{P}(v)$ and $\partial\mathrm{Per}(v)$: $p_{\widetilde{\pm}}$, $p_\pm$}
 Since any limit cycle in $\partial\mathrm{P}(v)$ (resp. $\partial\mathrm{Per}(v)$) is
  a boundary of a transverse annulus (resp. a periodic annulus), it is a periodic orbit and, at the same time, it has one or two attracting/repelling collars. 
 Hence, a limit cycle surrounded by periodic orbits outside necessarily has an attracting/repelling collar inside as shown in Figure~\ref{fig:Boundary_isolated_point}(c). 
 It thus belongs to 
 $\partial\mathrm{Per}(v) \cap \partial\mathrm{P}(v)$ and its COT symbol is given by $p_{\widetilde{+}}(\B_{b\widetilde{\pm}})$ (resp. $p_{\widetilde{-}}(\B_{b\widetilde{\pm}})$)
 when the limit cycle is going in the counter-clockwise (resp. clockwise) direction, in which $\B_{b_{\widetilde{\pm}}}$ represents an internal attracting/repelling 
 orbit structure  in class-$b_{\widetilde{\pm}}$. On the other hand, when a limit cycle is surrounded by non-closed orbits outside,  it has already been  an element of $\partial\mathrm{P}(v)$. Then the internal orbit 
 structure is a transverse annulus or a periodic annulus. The COT symbol of the counter-clockwise (resp. the clockwise) limit cycle of this type 
  is given by $p_+(\B_{b+})$ (resp. $p_-(\B_{b-})$), in which $\B_{b\pm}$ symbolizes an internal orbit structure of class-$b_\pm$. The orbit structures $p_\pm$ and $p_{\widetilde{\pm}}$ are elements of class-$\widetilde{\pm}$ and class-$\alpha_\pm$. 

\subsubsection{Non-trivial circuits in $\partial\mathrm{P}(v)$, $\partial\mathrm{Per}(v)$ and $\mathrm{P}_\mathrm{sep}(v)$}
A non-trivial circuit with ($\partial$-)saddles is a union of saddle separatrices and ($\partial$-)saddles as shown in Figure~\ref{Fig08:NLC}. 
Since a saddle is accompanied by four separatrices, local orbit structures have the following possibilities.
\begin{itemize}
\item[($S_1$)] One separatrix is connecting to a source structure, another one is going to a sink structure. The remaining two separatrices are self-connected.
\item[($S_2$)] The saddle is connected by two self-connected saddle separatrices.
\item[($S_3$)] The two separatrices are connecting to two sources (or sink) structures. 
\end{itemize}
In the third case, the remaining two saddle separatrices cannot be self-connected, hence no non-trivial circuit with an ($S_3$)-saddle exists. Regarding a $\partial$-saddle, since it has only one free separatrix, we have two possibilities of local orbit structures.
\begin{itemize}
\item[($S_4$)] The separatrix is connected to another $\partial$-saddle on the same boundary.
\item[($S_5$)] The separatrix is connecting to a source/sink structure.
\end{itemize}
A $\partial$-saddle satisfying ($S_4$) always forms a self-connected $\partial$-saddle separatrix on the boundary. We will discuss the case of ($S_5$) later. Based on the observation,
 we introduce the COT symbols of non-trivial circuits with saddles.

\begin{figure}
\begin{center}
\includegraphics[scale=0.40]{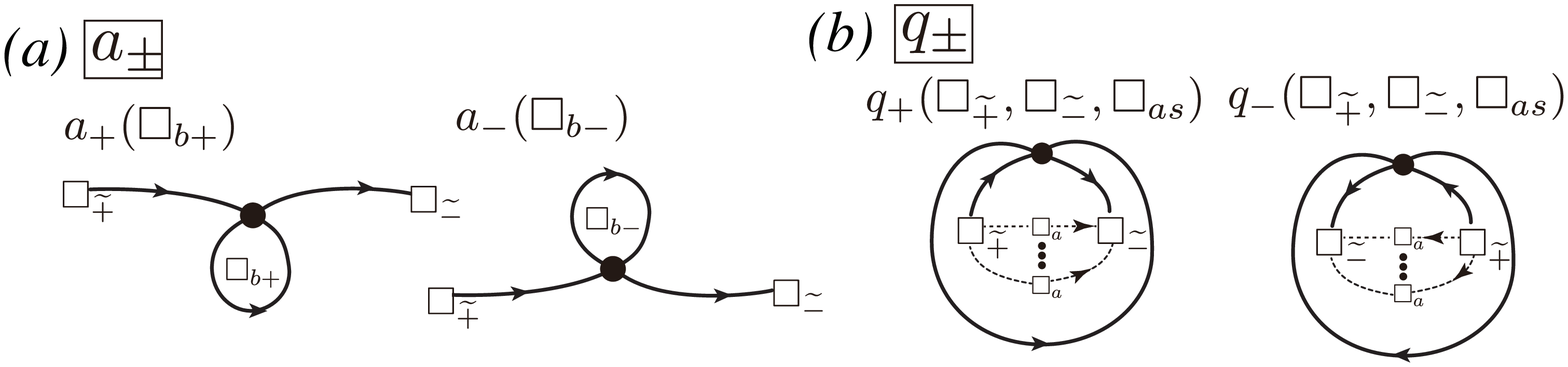}
\end{center}
\caption{Non-trivial circuits with an ($S_1$)-saddle. (a) The saddle has two ss-saddle separatrices connecting to a source structure and a sink structure outside the self-connected saddle separatrix, 
whose COT symbol is given by $a_\pm(\B_{b\pm})$.   (b) The saddle has two ss-saddle separatrices connecting to a source structure and a sink structure inside the self-connected 
saddle separatrix, whose COT symbol is $q_\pm(\B_{\widetilde{+}}, \B_{\widetilde{-}}, \B_{as} )$. The subscript $\pm$ is determined by the direction of the self-connected 
 saddle separatrix. They are in double-sign correspondence.}
\label{fig:S1_structures}
\end{figure}

{\bf [Non-trivial circuits  ($S_1$): $a_\pm$, $q_\pm$]} Figure~\ref{fig:S1_structures}(a) shows an obit structure with a self-connected saddle separatrix, outside of which two ss-saddle 
separatrices are connecting to a source structure and a sink structure. It is a class-$a$ orbit structure and its COT symbol is given by $a_+(\B_{b+})$ (resp. $a_-(\B_{b-})$) for 
the counter-clockwise (resp. the clockwise) saddle separatrix. The symbol $\B_{b+}$ (resp. $\B_{b-}$) indicates that  a class-$b_+$ (resp. a class-$b_-$) structure is enclosed
 by the self-connected saddle separatrix.

On the other hand, if ss-saddle separatrices connecting to a source structure and a sink structure exist inside of the self-connected saddle separatrix as shown in 
 Figure~\ref{fig:S1_structures}(b),
 its COT symbol is given by $q_+(\B_{\widetilde{+}}, \B_{\widetilde{-}}, \B_{as})$ (resp. $q_-(\B_{\widetilde{+}}, \B_{\widetilde{-}}, \B_{as})$) for the counter-clockwise
 (resp. the clockwise) saddle separatrix, in which $\B_{\widetilde{\pm}}$ and $\B_{as}$ denote the source and sink structures.
 The source-sink pair is connected by any number ($s\geq 0$) of class-$a$ orbit structures, which are symbolized by  
\begin{equation}
\B_{as} := \B_{a}^1 \cdot \cdots \cdot  \B_{a}^s  \quad (s>0), \qquad  \B_{as}:=\lambda_\sim \quad (s=0).
\label{COT_As}
\end{equation}
This orbit structure belongs to class-$\widetilde{\pm}$ and class-$\alpha_\pm$.

\begin{figure}
\begin{center}
\includegraphics[scale=0.40]{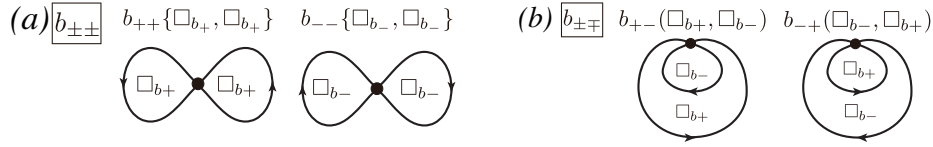}
\end{center}
\caption{Non-trivial circuits with an ($S_2$)-saddle. (a) Two saddle separatrices form a figure-eight structure whose COT symbol is given by $b_{\pm\pm} \{ \B_{b\pm}, \B_{b\pm} \}$. (b) An orbit structure where one self-connected saddle separatrix is enclosed by the other self-connected saddle separatrix. Its COT symbol is given by $b_{\pm\mp}(\B_{b\pm}, \B_{b\mp})$. All of them are in double-sign correspondence.}
\label{fig:S2_structures}
\end{figure}

{\bf [Non-trivial circuits ($S_2$): $b_{\pm\pm}$, $b_{\pm\mp}$]} When a non-trivial circuit with two self-connected saddle separatrices forms a figure-eight pattern as shown in 
  Figure~\ref{fig:S2_structures}(a), its COT symbol is represented by $b_{++}\{\B_{b+}, \B_{b+}\}$ or $b_{--}\{\B_{b-}, \B_{b-}\}$ depending on the direction of the saddle separatrices, in which 
 $\B_{b\pm}$ denotes class-$b_\pm$ orbit structure enclosed by the saddle separatrices. Since the order of the orbit structures $\B_{b_\pm}$ is not uniquely
 determined, they are enclosed by the parentheses $\{  \}$.

 Figure~\ref{fig:S2_structures}(b) shows a non-closed circuit with a saddle, in which a self-connected saddle separatrix is enclosed by another self-connected saddle 
 separatrix. Its COT symbol is $b_{+-}(\B_{b+}, \B_{b-})$ when the outer saddle separatrix is going in the counter-clockwise direction.
 Reversing the flow direction, we obtain a non-trivial circuit represented by $b_{-+}(\B_{b-}, \B_{b+})$. Inner class-$b_\pm$ orbit structures
 are denoted by $\B_{b\pm}$. They are elements of class-$\widetilde{\pm}$ and class-$\alpha_\pm$.

\begin{figure}
\begin{center}
\includegraphics[scale=0.39]{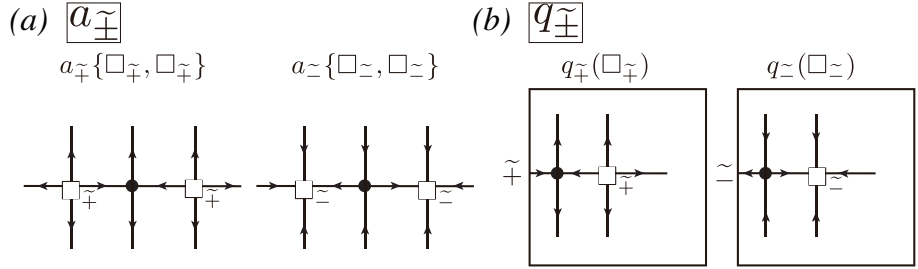}
\end{center}
\caption{Slidable saddles ($S_3$) associated with two source (or sink) structures. (a) Saddles and two source (sink) structures are in the same domain, whose COT symbols are given by $a_{\widetilde{\pm}}\{ \B_{\widetilde{\pm}}, \B_{\widetilde{\pm}}\}$. (b) A saddle and one source (sink) structure are contained inside the other source (sink) structure. Its COT symbols is given by $q_{\widetilde{\pm}}(\B_{\widetilde{\pm}})$. This orbit structure can be embedded only in a transverse annulus of class $b_{\widetilde{\pm}}$. All of them are in double-sign correspondence.}
\label{fig:S3_structures}
\end{figure}

{\bf [Slidable saddles ($S_3$): $a_{\widetilde{\pm}}$, $q_{\widetilde{\pm}}$]} A slidable saddle of ($S_3$) connecting to two source structures (resp. two sink structures)
 is represented by $a_{\widetilde{\pm}}$ (resp. $q_{\widetilde{\pm}}$). See Figure~\ref{fig:S3_structures}. Since the orbit structures are border orbits between trivial flow boxes, they belong to $\mathrm{P}_\mathrm{sep}(v)$. If the saddle and the two source (resp. sink) structures are in the same domain as in Figure~\ref{fig:S3_structures}(a), its COT symbol is  given by $a_{\widetilde{+}}\{ \B_{\widetilde{+}}, \B_{\widetilde{+}}\}$ (resp. $a_{\widetilde{-}}\{ \B_{\widetilde{-}},  \B_{\widetilde{-}}\})$. On the other hand, if the saddle and the one source (resp. sink) structure are contained in  the other source (resp. sink) structure as in Figure~\ref{fig:S3_structures}(b), the orbit structure is represented by $q_{\widetilde{+}}(\B_{\widetilde{+}})$  (resp. $q_{\widetilde{-}}(\B_{\widetilde{-}})$). Note that the orbit structure $q_{\widetilde{\pm}}$ is of class-$a_{\widetilde{\pm}}$ and it can be embedded only in a transverse annulus $b_{\widetilde{\pm}}$ in Figure~\ref{fig:Saturated_structures}(a). 
The flow boxes as in Figure~\ref{fig:S3_structures}(b) are called Cherry flow boxes. 

\begin{figure}
\begin{center}
\includegraphics[scale=0.40]{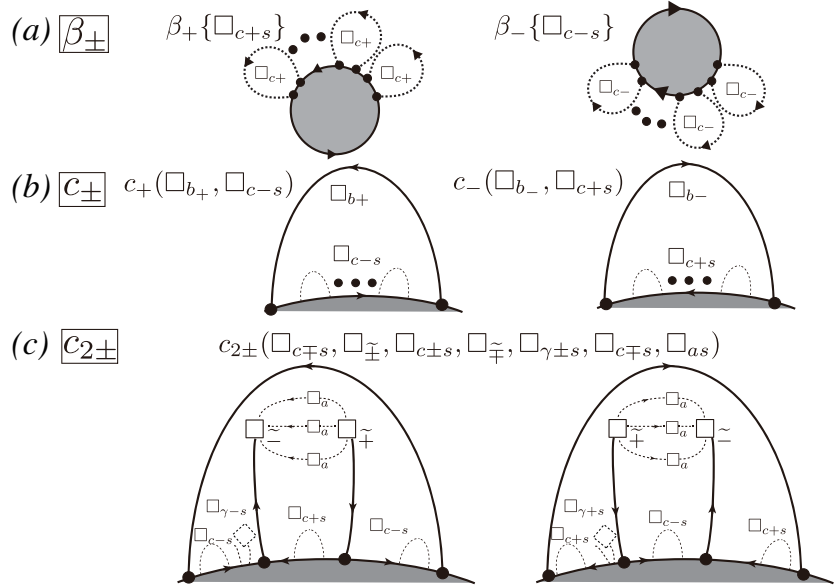}
\end{center}
\caption{Non-trivial circuits with ($S_4$)-saddles. (a) A solid boundary having self-connected $\partial$-saddle separatrices. When the flow direction along the boundary is counter-clockwise
 (resp. clockwise), the COT symbol is given by $\beta_+\{ \B_{c+s} \}$ (resp. $\beta_-\{ \B_{c-s} \}$). (b) When there contain no source-sink pairs in the 
domain enclosed by the $\partial$-saddle separatrix, the domain is filled with periodic orbits or non-closed orbits except a singular orbit and there are any number  of $\partial$-saddle separatrices inside, 
whose COT symbol is given by $c_\pm(\B_{b_\pm}, \B_{c \mp s})$. (c) An orbit structure containing slidable $\partial$-saddles inside a self-connected $\partial$-saddle separatrix attached
 to the boundary.
 Arranging the internal orbit structures from right to left, we have
  $c_{2\pm}(\B_{c\mp s}, \B_{\widetilde{\pm}}, \B_{c\pm s}, \B_{\widetilde{\mp}}, \B_{\gamma \mp s},\B_{c\mp s},  \B_{as})$ as its COT symbol. All of them are 
  in double-signs correspondence.}
\label{fig:S4_structures}
\end{figure}

{\bf [Non-trivial circuits ($S_4$): $\beta_\pm$, $c_\pm$, $c_2$]} If self-connected $\partial$-saddle separatrices are attached to the solid boundary, the flow along the boundary no longer belongs to $\partial_{\mathrm{per}}(v)$, but the orbit becomes a non-trivial circuit.   See Figure~\ref{fig:S4_structures}(a). 
 When the flow direction  is counter-clockwise (resp. clockwise), $\beta_+\{ \B_{c+s}\}$ (resp. $\beta_-\{ \B_{c-s} \}$) is assigned to the orbit structure.
 In the COT symbols, $\B_{c\pm s}$ is the abbreviation of 
\begin{equation}
\B_{c\pm s} := \B_{c\pm}^1\cdot \cdots \B_{c\pm}^s \quad (s>0), \qquad\B_{c\pm s}:=\lambda_\pm \quad (s=0),
\label{COT_Cs}
\end{equation}
in which $\B_{c\pm}^i$ ($1\leq i \leq s$) denotes the $i$th class-$c_\pm$ orbit structure arranged cyclically in the counter-clockwise direction.

On the other hand, we classify local orbit structures in the domain enclosed by the self-connected $\partial$-saddle separatrix and the solid boundary. 
When the domain contains no ss-separatrices between the boundary and source and sink structures, we obtain the orbit structures in Figure~\ref{fig:S4_structures}(b). The COT symbol becomes $c_+(\B_{b+}, \B_{c-s})$ (resp. $c_-(\B_{b-}, \B_{c+s})$) 
 when the flow direction of the $\partial$-saddle separatrix is counter-clockwise (resp. clockwise). In the domain, the two-dimensional orbit structure
 of class-$b_\pm$ is symbolized by  $\B_{b_\pm}$. In addition, any number ($s\geq 0$) of self-connected $\partial$-saddle separatrices with the opposite direction
 can be attached to the same boundary, which are represented by $\B_{c\mp s}$.

When ss-separatrices between the boundary and source and sink structures are contained in the domain, 
 we have an orbit structure with slidable $\partial$-saddles in a trivial flow box.
 Since we can attach any number of slidable $\partial$-saddles satisfying ($S_5$)
 to the boundary, we choose the rightmost 
 one as a special source-sink structure pair represented by $\B_{\widetilde{\pm}}$. The class-$a$ orbit structures connecting the source/sink structures are
 represented by $\B_{as}$ as in (\ref{COT_As}).
The other pairs of slidable $\partial$-saddles left to the special one are regarded as class-$\gamma_\pm$ orbit structures,  which are expressed by
 \[
\B_{\gamma \pm s}:= \B_{\gamma \pm}^1\cdot \cdots \cdot \B_{\gamma \pm}^s \quad (s>0), \qquad  \B_{\gamma \pm s}:=\lambda_\sim \quad (s=0).
\]
Between these source-sink pairs,  we can attach any number of $\partial$-saddle separatrices to the boundary, symbolized by $\B_{c\pm s}$. Arranging them 
 from the right, we obtain the COT representation 
 $c_{2+}(\B_{c-s}, \B_{\widetilde{+}}, \B_{c+s}, \B_{\widetilde{-}}, \B_{\gamma -s},\B_{c-s},  \B_{as})$ 
(resp. $c_{2-}(\B_{c+s}, \B_{\widetilde{-}}, \B_{c-s}, \B_{\widetilde{+}}, \B_{\gamma +s},\B_{c-s},  \B_{as})$) when the self-connected $\partial$-saddle separatrix goes in the counter-clockwise (resp. the clockwise) direction as in Figure~\ref{fig:S4_structures}(c). They are elements of class-$c_\pm$.  

\begin{figure}
\begin{center}
\includegraphics[scale=0.4]{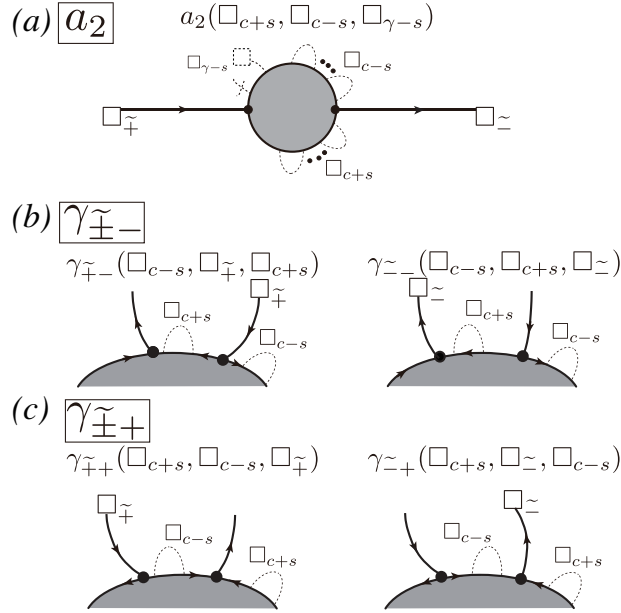}
\end{center}
\caption{Orbit structures with slidable $\partial$-saddles satisfying ($S_5$). (a) An orbit structure having a slidable $\partial$-saddle whose COT symbol
 is given by $a_2(\B_{c+s}, \B_{c-s}, \B_{\gamma-s})$. (b) Orbit structures with a pair of slidable $\partial$-saddles attached in the counter-clockwise direction, whose COT symbols are given by $\gamma_{\widetilde{+}-}(\B_{c-s}, \B_{\widetilde{+}}, \B_{c+s} )$ and $\gamma_{\widetilde{-}-}(\B_{c-s}, \B_{c+s}, \B_{\widetilde{-}})$.
 (c) Orbit structure with a pair of slidable $\partial$-saddles  attached in the counter-clockwise direction, whose COT symbols are given by $\gamma_{\widetilde{+}+}(\B_{c+s}, \B_{c-s}, \B_{\widetilde{+}})$ and $\gamma_{\widetilde{-}+}(\B_{c+s}, \B_{\widetilde{-}}, \B_{c-s})$.}
\label{fig:S5_structures}
\end{figure}

{\bf [Slidable $\partial$-saddles ($S_5$): $a_2$, $\gamma_{\widetilde{\pm}\pm}$]} 
 Since a pair of $\partial$-saddles satisfying ($S_5$)  connects to a source and a sink structure,  we have a slidable $\partial$-saddle. 
 Since we can attach any number of pairs of slidable $\partial$-saddles  to the solid boundary, we choose a special one so that no slidable 
 $\partial$-saddles exist in the lower part of the boundary, which yields an orbit structure as in Figure~\ref{fig:S5_structures}(a). 
  By construction,  any number of  class-$c_+$ orbit structures are attached to the lower part of the boundary, whereas any number of class-$\gamma_-$ and class-$c_-$ orbit structures exist on the upper part. Hence, reading those structures in the counter-clockwise direction along the boundary, we provide the COT symbol $a_2(\B_{c+s}, \B_{c-s}, \B_{\gamma-s})$. It is a class-$a$ orbit structure.

 The orbit structures with an ($S_5$) $\partial$-saddle are all slidable $\partial$-saddles. 
Suppose that slidable $\partial$-saddles are attached to the boundary in $\beta_{\emptyset 2}$ of Figure~\ref{fig:Fundamental_structures}(c), $c_{2\pm}$ of Figure~\ref{fig:S4_structures}(c), and $a_2$ of Figure~\ref{fig:S5_structures}(a). 
 If an attached slidable $\partial$-saddle is of Figure~\ref{Fig05:SS}(c) and Figure~\ref{fig:S4_structures}(c) (resp. of Figure~\ref{fig:S4_structures}(b)), the COT symbol is given by $\gamma_{\widetilde{+}+}(\B_{c+s}, \B_{c-s}, \B_{\widetilde{+}})$ and $\gamma_{\widetilde{-}+}(\B_{c+s}, \B_{\widetilde{-}}, \B_{c-s})$ (resp. $\gamma_{\widetilde{+}-}(\B_{c-s}, \B_{\widetilde{+}}, \B_{c+s} )$ and $\gamma_{\widetilde{-}-}(\B_{c-s}, \B_{c+s}, \B_{\widetilde{-}})$). 
They are class-$\gamma_\pm$ orbit structures.

\subsection{Decomposition of flows of finite type}
 Let $G$ denote the union of non-trivial circuits in $\mathrm{Bd}(v)$ such as cycles and orbits with ($\partial$-)saddle separatrices in Figure~\ref{Fig08:NLC}.
 Precisely, in terms of flow components with COT symbols, it consists of cycles $p_\pm$ and $p_{\widetilde{\pm}}$, orbits having a self-connected saddle separatrix
 $a_\pm$ and $q_\pm$, orbits having two self-connected saddle separatrices $b_{\pm\pm}$ and $b_{\pm\mp}$, orbits generated by $a_2$ with self-connected 
 $\partial$-saddle separatrices $c_\pm$ and $c_{2\pm}$, and orbits along solid boundaries $\beta_\pm$ with/without self-connected $\partial$-saddle separatrices. 
  According to Theorem~\ref{thm:flowbox}, any non-trivial circuit in $\mathrm{Bd}(v)$ is a boundary of trivial flow boxes, transverse annuli,  and periodic annuli. 
 We then introduce the following subset of $G$.
 \begin{definition}
 $\Gamma_p$ denotes the union of non-closed orbits $\gamma$, satisfying 
 \begin{itemize}
 \item $\gamma \subset \Gamma_p$ is a border orbit in the local orbit structures with the COT symbol, $b_{\pm\pm}$, $b_{\pm\mp}$, or $\beta_\pm$ with/without self-connected $\partial$-saddle separatrices,
 \item all flow components  in $(\mathrm{Bd}(v))^c$ separated by $\gamma \subset \Gamma_p$ are periodic annuli.
 \end{itemize}
 \end{definition}
The following operation decomposes the flow domain $S$ into several disjoint open domains whose boundary orbits are
 non-trivial circuits in $\Gamma=G\setminus \Gamma_p$. 
 \begin{definition}(Metric completion) For a flow of finite type $v$ on a surface $S$,
The metric completion of the domain $S_0=S\setminus \Gamma$ for a flow of finite type $v$ on a surface $S$ is the completion of each domain in $S_0$ with keeping their metric unchanged.
\end{definition}
 Let $S_{\mathrm{me}}$ denote  the metric completion of $S_0$. Since $S_0$ consists of several open domains separated by $\Gamma$, 
 $S_{\mathrm{me}}$ becomes the union of the completed closed domains by definition. For convenience, the new boundary set of $S_{\mathrm{me}}$ is denoted by $\partial$.
 Then, we can define a canonical map $\pi: S_{\mathrm{me}} \to S$ such that $\pi\vert_{S_0}$ is the identity 
 and $\pi \vert_\partial \colon \partial \rightarrow \Gamma$ becomes an immersion for the border orbits in $\Gamma$. 
 While the inverse map $\pi\vert_{S_0}^{-1}$ is well-defined,  the map $\pi$ can be reconstructed from the image as long as the composition of the disjoint open domains in $S_0$ is known.
 In addition, the map $\pi\vert_{S_0}^{-1}$  induces a flow $v_{\mathrm{me}}$ on $S_{\mathrm{me}}$, which is identical to
 the original flow $v$ on $S_0$, and the flow orientation along the non-singular orbits in $\Gamma$ coincides with that along non-singular orbits in the boundary $\partial$.

 Next, we define a map collapsing each boundary component in $\partial$ into a singular orbit.
\begin{definition}(Boundary collapse) 
A map $\phi \colon S_{\mathrm{me}} \to S_{\mathrm{mc}}$ is defined by a surjective collapsing each boundary component in $\partial \subset S_{\mathrm{me}}$ to a singular orbit. 
\end{definition} 
By construction,
 the flow $v_{\mathrm{mc}}$ contains no limit circuits and the restriction of the flow on $\phi(\pi^{-1}(S_0))$ is topologically equivalent 
 to the flow $v$ on $S_0$.  We note that the inverse of $\phi$ is reconstructed from the image
  if we know the information on the boundary components in $\Gamma$ and the composition of the domains separated by the boundaries. Let us again recall that
 the flow $v_{\mathrm{me}}$ in the collar of $\gamma \subset \partial$  is either a transverse
 annulus, a periodic annulus, or a trivial flow box. If $\gamma \subset \partial$ is the boundary of a transverse annulus, the singular orbit $\phi(\gamma)$ becomes a source/sink, which gives rise to a gradient flow with non-degenerate singular
 points on $S_{\mathrm{mc}}$.  When $\gamma \subset \partial$ is the boundary of  a periodic annulus, the singular orbit $\phi(\gamma)$ 
 becomes a  center.  The restriction of $v_{\mathrm{mc}}$ is then a Hamiltonian flow, in which all singular orbits and all separatrices are self-connected.
 Hence, by Theorem 2.3.8, p. 74 \cite{MaWa05}, it is topologically equivalent to a structurally stable Hamiltonian flow on $S_{\mathrm{mc}}$. Finally when the flow
  in the collar of $\gamma \subset \partial$ is a trivial flow box, $\phi(\gamma)$ becomes a degenerate singular orbit. 
  More specifically, when $\pi(\gamma)$ is the self-connected
  saddle connection, which is the boundary orbit of the outer (resp. the inner) domain of orbit structure $a_\pm$ (resp. $q_\pm$), or is the boundary orbit of the outer domain
  of orbit structure $a_2$ with self-connected $\partial$-saddle separatrices, $\phi(\gamma)$ becomes a $0$-saddle. On the other hand, when $\pi(\gamma)$ is the
   boundary orbit of the inner domain 
  of an orbit structure $c_{2\pm}$, $\phi(\gamma)$ becomes a multi-saddle whose multiplicity is determined by the number of sources/sinks contained in the domain.
  In this case, the flow $v_{\mathrm{me}}$ on the domain induced by the map $\phi$ is a gradient flow with a degenerate singular orbit.
   Note that, in both cases, the flow $v_{\mathrm{mc}}$ has no homoclinic
  saddle separatrices, since any trivial box cannot contain saddles in its interior. Accordingly, the argument has shown the following decomposition of the flow of finite type.

\begin{proposition}\label{pr:decomposition}
Let $v$ be a flow of finite type $v$ on a surface $S$, and $v_{\mathrm{mc}}$ be the flow on the surface $S_{\mathrm{mc}}$ induced by the metric completion and the boundary 
collapse.
Then there exists a unique decomposition $S_{\mathrm{mc}} = S_g \sqcup S_H$, in which $S_g$ and $S_H$ are disjoint unions of connected components of $S_{\mathrm{mc}}$ 
such that $v_{\mathrm{mc}}\vert_{S_g}$ are gradient flows with non-degenerate singular orbits/multi-saddles and no homoclinic separatrices, 
and $v_{\mathrm{mc}}\vert_{S_H}$ are structurally stable Hamiltonian flows up to topological equivalence. 
\end{proposition}

\begin{figure}
\begin{center}
\includegraphics[scale=0.45]{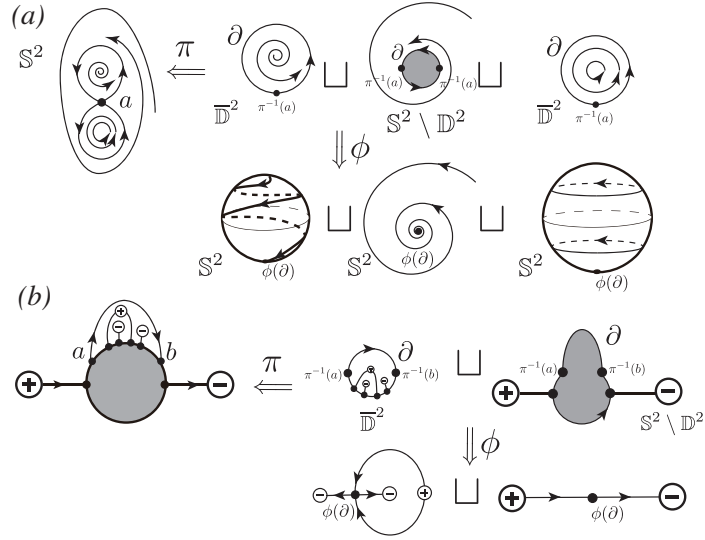}
\end{center}
\caption{(a) The images of  an orbit structure with two self-connected saddle connections ($b_{++}$) by the maps $\pi$ and $\phi$.
(b) The images of a non-trivial circuit ($a_2$) with a self-connected $\partial$-saddle ($c_{2-}$) enclosing two sinks and one source.}
\label{fig:linking_example}
\end{figure}

For instance, suppose that a non-trivial circuit $\gamma \subset \Gamma_g$ is a border orbit with the COT symbol $b_{++}$ among a transverse annulus outside, a transverse, and a periodic annulus inside as shown in Figure~\ref{fig:linking_example}(a).
  By the metric completion, the flow on $\mathbb{S}^2\setminus \gamma$ is then
divided into an attracting transverse annulus around the boundary on $\mathbb{S}^2\setminus\mathbb{D}^2$, a repelling transverse annulus in a closed disk $\overline{\mathbb{D}}^2$
and a periodic annulus on a closed disk $\overline{\mathbb{D}}^2$. By definition, the pre-images $\pi^{-1}(a)$ are contained in every boundary
of the domains as $\partial$-0-saddles. The boundary collapse $\phi$ gives rise to gradient flows consisting of sources and sinks on the sphere $\mathbb{S}^2$ and a 
structurally stable Hamiltonian flow filled with periodic orbits around two centers on $\mathbb{S}^2$. Figure~\ref{fig:linking_example}(b) shows
an orbit structure represented by $a_2$ with a self-connected $\partial$-saddle separatrix $c_2$ enclosing two sinks and one source structures.
The self-connected $\partial$-saddle separatrix $\gamma$ between the $\partial$-saddles at $a$, $b$ separates the flow domain into two trivial
 flow boxes. By the metric completion, the flow 
inside of $\gamma$ becomes a flow on a closed disk $\overline{\mathbb{D}}^2$ having  slidable $\partial$-saddles and the pre-images $\pi^{-1}(a)$
and $\pi^{-1}(b)$ as $\partial$-$0$-saddles. On the other hand, the flow outside of $\gamma$ is converted into a trivial flow box
whose boundary has the $0$-saddles $\pi^{-1}(a)$ and $\pi^{-1}(b)$. These flows are finally mapped into gradient flows with a $1$-saddle  and a $0$-saddle on $\mathbb{S}^2$
by the boundary collapse.

\begin{figure}
\begin{center}
\includegraphics[scale=0.5]{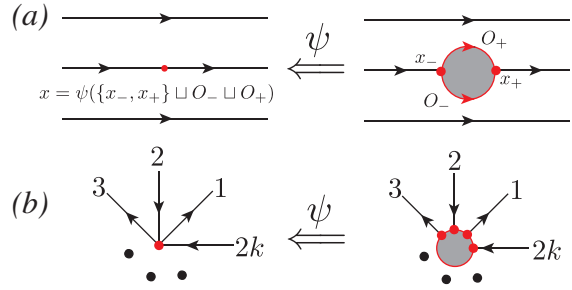}
\end{center}
\caption{(a) A blowdown map $\psi$ for a flow with one $0$-saddle $x$. It is defined by replacing a boundary component $\{x_-, x_+ \} \sqcup O_- \sqcup O_+$ consisting of two $\partial$-saddles $x_-, x_+$ and two separatrices $O_-$, $O_+$ between them with the $0$-saddle $x$. (b) A blowdown map $\psi$ for a flow with one $2k$-saddle. It replaces a boundary component which consists of $2(k+1)$ $\partial$-saddles and $2(k+1)$ separatrices with the $2k$-saddle}
\label{fig:blowup}
\end{figure}
To obtain the topological classification of flows of finite type, it is necessary to convert the set of gradient flow components on $S_g$ into that of Morse flows by 
{\it smoothing out} the degenerate singular orbits with the following operations.
\begin{definition} (Blowup/blowdown)
Let $A \subset \mathrm{int}S$ denote the set of multi-saddles that are not saddles for a flow $v$ on a surface $S$. Then a flow $w$ on a surface $T$ is a {\it blowup} of $v$ with respect to multi-saddles in $A$, if there are a surjection $\psi \colon T \to S$ and  a union of boundary components $B$ of $T$ consisting of $\partial$-saddles and  separatrices 
satisfying the following conditions.
\begin{itemize}
\item  $\psi\vert_{T \setminus B}: T \setminus B \to S \setminus A$ is a homeomorphism, and $v\vert_{S \setminus A}$ is topological equivalent to $w\vert_{T \setminus B}$ by $\psi$;
\item the inverse image of any degenerate multi-saddle of $v$ by $\psi^{-1}$ is a boundary component in $B$. 
\end{itemize}
\end{definition}
 The map $\psi\colon S_{\mathrm{mb}} \to S_{\mathrm{mc}}$ and its inverse $\psi^{-1}$ are well-defined, which are referred to as the {\it blowdown} and the {\it blowup} of $v$
 respectively. As shown in Figure~\ref{fig:blowup}, the blowup operation replaces 
 each $k$-saddle ($k \neq 1 \in \mathbb{Z}_{\geq 0}$) with a boundary component consisting of $2(k+1)$ $\partial$-saddles and $2(k+1)$ free separatrices.
 Gradient flows with degenerate saddles are converted into Morse flows using the blowup. We thus show the decomposition theorem of a flow of finite type on a surface into structurally
 stable Hamiltonian flows and Morse flows. 

\begin{theorem}\label{th:decomposition}(Decomposition theorem)
 Let $v_{\mathrm{mb}}$ be a flow on the surface $S_{\mathrm{mb}}$ induced by $\psi^{-1}\circ \phi\circ \pi^{-1}$ from the flow of finite type $v$ on a surface $S$.
 Then, there exists a unique decomposition $S_{\mathrm{mb}} = S_M \sqcup S_H$, in which $S_M$ and $S_H$ are disjoint unions of the connected components of $S_{\mathrm{mb}}$,
 such that $v_{\mathrm{mb}}|_{S_M}$ are Morse flows and $v_{\mathrm{mb}}|_{S_H}$ are structurally stable Hamiltonian flows. 
\end{theorem}

\begin{proof}
Proposition~\ref{pr:decomposition} implies that there is a unique decomposition $S_{\mathrm{mc}} = S_g \sqcup S_H$ such that $v_{\mathrm{mc}}\vert_{S_g}$ are gradient flows with 
 non-degenerate singular orbits or with multi-saddles but no homoclinic separatrices, and $v_{\mathrm{mc}}\vert_{S_H}$ are structurally
 stable Hamiltonian flows. By definition, the blowdown $\psi \colon S_{\mathrm{mb}} \to S_{\mathrm{mc}}$ is the identity map on $S_H$. By the blowup of flows on $S_g$ with 
 multi-saddles in $\mathrm{int} S_g$, the restriction to $\psi(S_g)$ of the flow $v_{\mathrm{mb}}$ have $\partial$-saddles but no self-connected saddle separatrices in $\mathrm{int} S$.
  This implies that flow $v_{\mathrm{mb}}$ on $S_M:=\psi^{-1}(S_g)$ consists of Morse flows. We thus have the decomposition $S_{\mathrm{mb}} = S_M \sqcup S_H$ as desired. 
\end{proof}

\subsection{Discrete representations of topological structures for flows of finite type}
According to Theorem~\ref{th:decomposition}, the flow $v_{\mathrm{mb}}$ on $S_{\mathrm{mb}} = S_M \sqcup S_H$ is decomposed into Morse flow components $v_M$ on $S_M$
 and structurally stable Hamilton flow components $v_H$ on $S_H$. Smale~\cite{S61} proved that Morse flows on closed surfaces are gradient flows without separatrices from a 
 saddle to a saddle, which essentially says that the ss-saddle connection diagram is uniquely assigned to any Morse flow $v_M$. Hence, we characterize the topological orbit 
 structure of any Morse flow components $v_M$ with $L_k(v_M):=D_{ss}(v_M)$. We thus refer to $L_k(v_M)$ as the {\it linking structure}  of $v_M$,
 which are equivalently expressed as surface graphs. On the other hand, it has been shown in \cite{SY18,UYS18} that
  the topological orbit structures of structurally stable Hamiltonian flows are in one-to-one correspondence with the COT representations.
  Hence, a unique combinatorial tree/graph representation of the topological orbit structure for a given flow of finite type $v$ on a surface $S$ can be obtained, in principle, 
  from the linking structures and the COT representations for the flow $v_{\mathrm{mb}}=v_M \sqcup v_H$ induced from $v$ by $\pi\circ \phi^{-1} \circ \psi$.
  However, as discussed in Section~4.2, the maps $\pi$ and
  $\phi^{-1}$ cannot be constructed unless the composition of non-trivial circuits and their neighboring 2D domains are known. We thus extend the COT representation developed for structurally stable Hamiltonian flows to that for flows of finite type to retain the geometric information in the previous section. This is accomplished by providing a conversion algorithm from a flow of finite type into a unique COT representation with a linking structure.

\subsubsection{Conversion algorithm and combinatorial classification theorem}
Let us define a procedure applicable to any local orbit structure $v_p$ having a COT symbol, say $Cot(v_p)$, in Figures~\ref{fig:Saturated_structures}--\ref{fig:S5_structures}. 
 \begin{description}
 \item[Step 1] We identify inner orbit structures, $v_c:=\{ v_c^1, \ldots, v_c^n\}$, embedded in $v_p$ according to the square symbols in $Cot(v_p)$ and Table~\ref{tbl:COT_Structures};
 \item[Step 2] Regarding $v_p$ as a parent node and $v_c^i$, $i=1,\dots, n$ as  child nodes, we connect them by directed edges;
 \item[Step 3] We replace the square symbols in $Cot(v_p)$ with  $Cot(v_c^i)$, $i=1,\dots, n$.
  \end{description}
 Using this procedure, we describe an algorithm providing a unique COT representation to the ss-saddle connection diagram of a given flow of finite type $v$:
 The first step is identifying the root structure of the flow, to which the above-mentioned procedure is applied. We then repeat the procedure recursively to all inner 
 orbit structures until they reach innermost orbit structures such as centers, sources/sinks, and solid boundaries without 
 $\partial$-saddle separatrices. The algorithm certainly terminates owing to the finiteness of orbit structures in the flow of finite type, and it thus
  yields a unique rooted and directed tree represented by a combination of COT symbols. Let us write $Cot(v)$ as the 
 output of this algorithm for a given flow of finite type $v$. Since $Cot(v)$ retains the composition of all orbit structures in $v$ making the maps $\pi$ and $\phi$ well-defined,
  we have the following combinatorial classification theorem for the flows of finite type.

\begin{theorem}\label{th:COT}(Combinatorial classification theorem)
Any flow of finite type $v$ on a surface $S$ is uniquely represented by a pair of the COT representation $Cot(v)$ and the linking structure $L_k(v_M)$, in which $v_M$ denotes the Morse components of $v_{\mathrm{mb}}$ on $S_M$ induced by $\psi^{-1}\circ \phi \circ \pi^{-1}$.
\end{theorem}

\begin{figure}
\begin{center}
\includegraphics[scale=0.55]{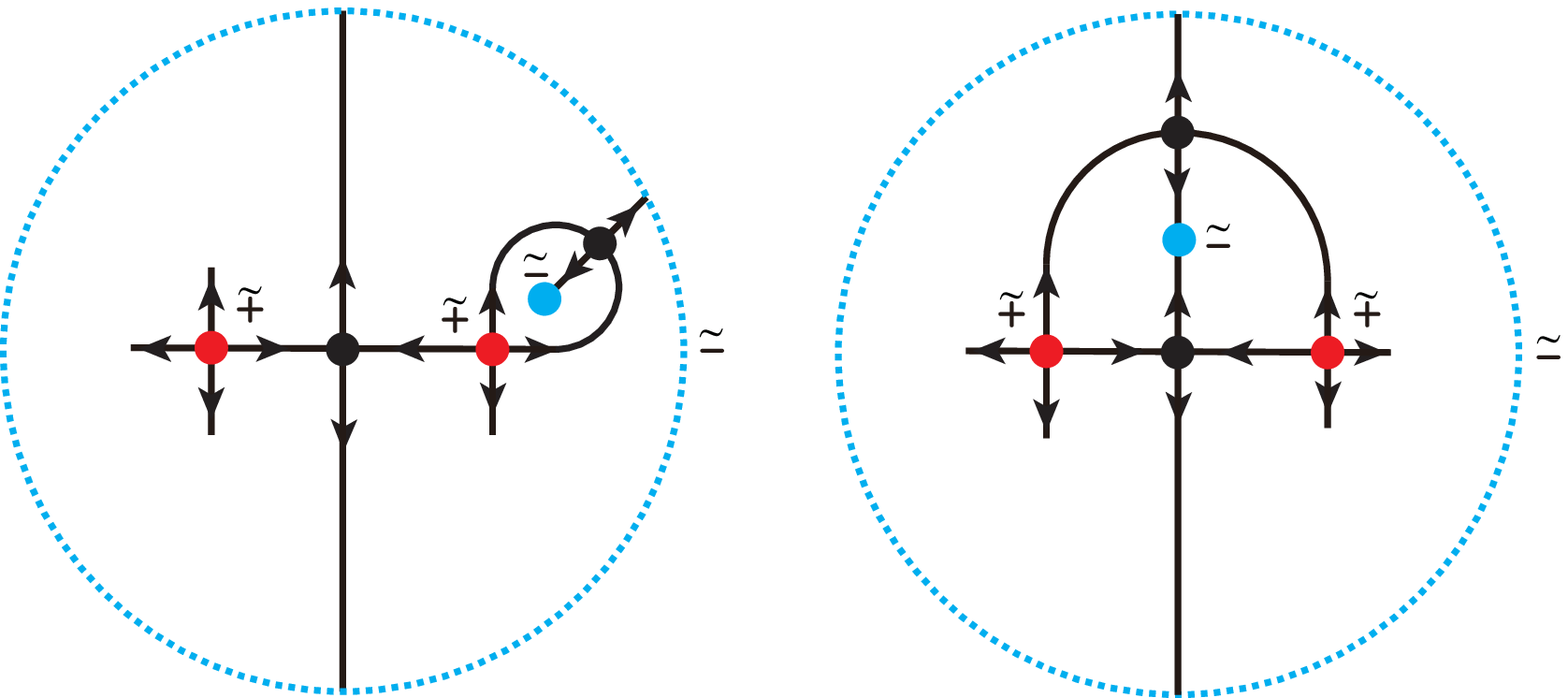}
\end{center}
\caption{Orbit structures with the same COT representation, $\sigma_{\emptyset\widetilde{-}0}(b_{\widetilde{+}}( a_{\widetilde{+}} \{\sigma_{\widetilde{+}0}, \sigma_{\widetilde{+}0}\}, \{ q_{\widetilde{+}}(\sigma_{\widetilde{-}0}) \}))$, but different linking structures.}
\label{fig:Combinatrics-examples}
\end{figure}

 Any structurally stable Hamiltonian flow is uniquely identified with the COT representation only, but this is not the case with the flow of finite type.
 In other words, the linking structure is also a substantial identifier, since there exist different orbit structures with the same COT representation 
as shown in Figure~\ref{fig:Combinatrics-examples}.

\section{Examples}
\label{section:Example}
\subsection{Applications to 2D compressible flows}
The first example is the flow generated by the Van del Pol equation, $\dot{x} = y-x^3+x$ and $\dot{y}=-x$.
Orbit structures generated by the solutions are shown in  Figure~\ref{fig:ss-examples}(a) to which we assign the COT representation. Since the flow domain contains
 no solid boundaries and the point at infinity is a 
source associated with counter-clockwise non-closed orbits, the root structure is $\sigma_{\emptyset\widetilde{+}+}(b_{\widetilde{-}}(\B_{\widetilde{-}},\{\B_{a\widetilde{-}s}\}))$.
 In $\B_{\widetilde{-}}$, a clockwise limit cycle having the COT symbol $p_-(\B_{b-})$ is embedded. No class-$a$ structures exist, i.e., $\B_{a\widetilde{-}s}=\lambda_\sim$.
  Since the domain inside the limit cycle is a repelling transverse annulus from a source at the origin surrounded by clockwise non-closed orbits to the limit cycle, 
  we have $\B_{b-} = b_{\widetilde{+}}(\sigma_{\widetilde{+}-},\{ \lambda_\sim \})$. Consequently, we obtain 
 $\sigma_{\emptyset\widetilde{+}+}(b_{\widetilde{-}}(p_-(b_{\widetilde{+}}(\sigma_{\widetilde{+}-},\{\lambda_\sim\})),\{\lambda_\sim\}))$ as the COT representation. 
 
 The second examples are potential flows with point vortices and source/sink pairs in a bounded domain. 
 The flow domain is inside the unit circle $D_\zeta=\{ \zeta \in \mathbb{C}\, \vert \, \vert \zeta \vert \leqq 1 \}$ containing three solid circular
  boundaries $C_m$,  $m=1,2,3$, whose centers $\delta_m \in \mathbb{C}$ and radii $q_m \in \mathbb{R}$ are given by
\[
\delta_1 = 0.052+0.465i, \; q_1 = 0.15, \; \delta_2 = -0.280-0.628i, \; 
\]
\[
q_2 = 0.18, \; \delta_3 = 0.483-0.281i, \; q_3 = 0.21.
\]
Analytic formulae of the complex potentials of a point vortex and a source-sink pair are described by using a transcendental function called the Schottky Klein prime 
function associated with the circular domain $D_\zeta$~\cite{DC05,DC11}. These complex potentials are computed numerically with the \textit{SKPrime} package~\cite{DGKM16}
on \textsc{Matlab}.

 Figure~\ref{fig:ss-examples}(b) shows streamlines of a potential flow consisting of seven-point vortices. Since the flow contains no ss-components, it
 is also a structurally stable Hamiltonian flow. Let us assign the COT representation to the flow to check the consistency with the existing theories~\cite{SY18,YS12}. We list the locations $v_m$ and the strengths $\kappa_m$ of
 the point vortices for $m=1, \dots, 7$  in Table~\ref{tbl:configurations} for reference. Note that a point vortex with the positive (resp. the negative) strength generates 
 a counter-clockwise (resp. a clockwise) flow in its neighborhood. The root structure of this
 flow is $\beta_{\emptyset-}$ since the flow along the outer boundary is going in the counter-clockwise direction. Hence, its COT representation starts with 
$\beta_{\emptyset-}(\B_{b+}^1, \{ \B_{c-s}^1\} )$. In $\B^1_{b+}$, we embed a periodic annulus whose boundary is a streamline connecting the solid
 boundary $C_2$. The internal orbit structures embedded in  $\B_{c-s}^1$ are two clockwise $\partial$-saddle separatrices attached to the outer boundary, enclosing the two-point vortices
 $v_1$ and $v_2$ with the strengths $\kappa_1=-1$ and $\kappa_2=-0.5$. We thus have 
 \[
\B_{b+}^1 = b_+(\beta_+\{ \B_{c+s}^2 \}), \qquad \B_{c-s}^1 = c_-(b_-(\sigma_-), \lambda_+) \cdot c_-(b_-(\sigma_-), \lambda_+).
 \]
The boundary $C_2$ has three self-connected $\partial$-saddle separatrices. The first one is enclosing the point vortex $v_3$ with the  strength $\kappa_3=0.5$ on its righthand side, and 
the second one contains a streamline connecting the solid boundary $C_1$ and the third clockwise $\partial$-saddle connection is attached to the boundary $C_2$ 
around the point vortex $v_4$ with the  strength $\kappa_4=-0.5$. Hence, we obtain
 \[
 \B_{c+s}^2 =c_+(b_+(\sigma_+), \lambda_-) \cdot c_+(b_+(\beta_+\{\B_{c+s}^3 \})) \cdot c_-(b_-(\sigma_-), \lambda_+).
 \]
 In addition, the $\partial$-saddle separatrix on the solid boundary $C_1$ contains a figure-eight $b_{++}$ orbit structure, in which we embed a $b_{+-}$ orbit structure
 enclosing the two point vortices $v_5$ and $v_6$, and a counter-clockwise $\partial$-saddle separatrix on the solid boundary $C_3$ enclosing the point
  vortex $v_7$ with the  strength $\kappa_7=1$.  The COT representation of the streamline pattern in Figure~\ref{fig:ss-examples}(b) finally becomes
\begin{eqnarray*}
&& \beta_{\emptyset-}(b_+(\beta_+\{ c_+(b_+(\sigma_+), \lambda_-) \cdot c_+(b_+(\beta_+\{\B_{c+s}^3 \})) \cdot c_-(b_-(\sigma_-), \lambda_+)\}), \\
&& \{ c_-(b_-(\sigma_-), \lambda_+) \cdot c_-(b_-(\sigma_-), \lambda_+)\} ),
 \end{eqnarray*}
in which $\B_{c+s}^3 = c_+(b_+(b_{++}\{ b_+(b_{+-}(b_+(\sigma_+), b_-(\sigma_-)), b_+(\beta_+\{ c_+(b_+(\sigma_+), \lambda_-) \} )\}))$.

\begin{figure}
\begin{center}
\includegraphics[scale=0.42]{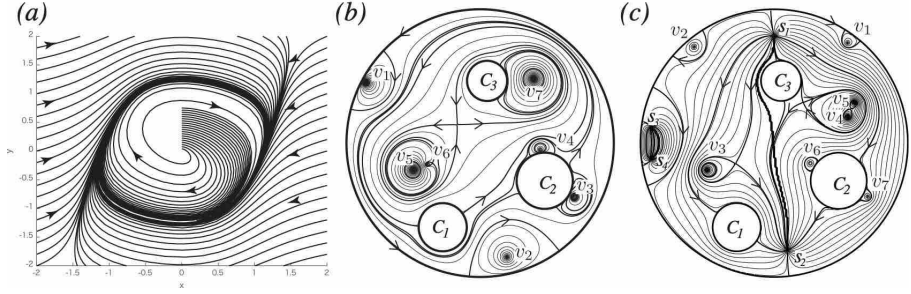}
\end{center}
\caption{(a) Orbit structure generated by the Van-del-Pol equation. (b) Streamlines of a Hamiltonian vector field with seven point vortices at $v_m$, $m=1,2, \dots, 7$ 
in the unit circle. (b) Streamlines of a flow of finite type with two source-sink pairs  $(s_1, s_2)$ and $(s_3, s_4)$ and seven point vortices $v_m$, $m=1,2,\dots, 7$ in 
the unit circle. The locations and the strengths of the point vortices and the source-sink pairs are listed in Table~\ref{tbl:configurations}.}
\label{fig:ss-examples}
\end{figure}

\begin{table}
\caption{Locations and strengths of point vortices and source-sink pairs  in Figure~\ref{fig:ss-examples}(b, c).}
\begin{center}
\begingroup
\scalefont{0.89}
\begin{tabular}{|l|l||l|l|}\hline
Fig.~\ref{fig:ss-examples}(b)  & Strength & Fig.~\ref{fig:ss-examples}(c)  &Strength \\ \hline\hline
$v_1=-0.85+0.45i$ & $\kappa_1=-1$ & $v_1=0.55+0.75i$ & $\kappa_1=0.5$ \\ \hline
$v_2=0.20-0.85i$ & $\kappa_2=-0.5$  & $v_2=0.60+0.72i$ & $\kappa_2=-0.5$\\ \hline
$v_3=0.70-0.40i$ & $\kappa_3=0.5$ & $v_3=-0.5-0.2i $ & $\kappa_3=1$\\ \hline
$v_4=0.45-0.05i$ & $\kappa_4=-0.5$  & $v_4=0.55+0.20i$ & $\kappa_4=-1.0$\\ \hline
$v_5=-0.50-0.20i$ & $\kappa_5=1.0$  & $v_5=0.60+0.30i$ & $\kappa_5=-1.0$\\ \hline 
$v_6=-0.40-0.16i$ & $\kappa_6=-0.25$ & $v_6=0.27-0.15i$ & $\kappa_6=0.5$ \\ \hline
$v_7=0.40+0.48i$ & $\kappa_7=1$ & $v_7=0.70-0.40i$ & $\kappa_7=-0.5$ \\ \hline
& & $s_1=-0.00+0.80i$, $s_2=0.10-0.80i$ & $m_1=1$ \\ \hline
& & $s_3=0.92+0.12i$, $s_4=-0.92-0.12$& $m_2=1$\\ \hline
\end{tabular}
\endgroup
\end{center}
\label{tbl:configurations}
\end{table}

Let us next consider a flow of finite type having source-sink pairs whose root structure is  $\beta_{\emptyset 2}$. Figure~\ref{fig:ss-examples}(c) is the streamlines 
 of a potential flow with seven point vortices  $v_m$ with strengths $\kappa_m$, $m=1,2,\dots, 7$  and two source-sink pairs $(s_1, s_2)$ and $(s_3, s_4)$ with strength one.
 Their locations and strengths are listed in Table~\ref{tbl:configurations}.  The COT representation of this flow is initially given by
\[
 \beta_{\emptyset 2} (\{ \B_{c+s}^1, \B_{\widetilde{-}}^1, \B_{c-s}^1, \B_{\widetilde{+}}^1, \B_{\gamma + s}^1 \}, \B_{as}^1 ).
\]
 We then determine the internal structures corresponding the square symbols. Since there exist the source  $s_1$ associated with non-rotating non-closed orbit and the sink $s_2$ in the root
  structure, we have $\B_{\widetilde{+}}^1=\sigma_{\widetilde{+}0}$
 and $\B_{\widetilde{-}}^1 = \sigma_{\widetilde{-}0}$. On the right hand side of the source-sink pair, there is one self-connected $\partial$-saddle separatrix
  enclosing the point vortex $v_1$ with the  
 strength $\kappa_1=0.5$, and no class-$\gamma_+$ orbit structures exist. Hence it yields $\B_{c+s}^1=c_+(b_+(\sigma_+), \lambda_-)$ and 
 $\B^1_{\gamma + s} = \lambda_\sim$. On the lefthand side, we see two clockwise $\partial$-saddle connections attached to the outer boundary. 
 One is enclosing the point vortex $v_2$ with the  strength $\kappa_2=-0.5$, and the other one contains the source-sink pair $(s_3, s_4)$ without any special orbit structure
  belonging to class-$c_{\pm}$, class-$\gamma_+$ and class-$a$. These structures are thus represented by 
 \[
 \B_{c-s}^1 = c_{2-}( \lambda_+, \sigma_{\widetilde{-}0},  \lambda_-, \sigma_{\widetilde{+}0}, \lambda_\sim,  \lambda_+, \lambda_\sim ) \cdot c_-(b_-(\sigma_-), \lambda_+).
 \]
 We finally identify class-$a$ orbit structures in $\B_{as}^1$ connecting the source-sink pair $(s_1, s_2)$. Checking the structures from the left, we find the orbit 
 structure $a_2$ connecting the solid boundary $C_1$,  $a_+$ enclosing the point vortex $v_3$ with the  strength $\kappa_3=1$, $a_2$ connecting the solid boundary $C_3$, $a_-$ enclosing a 
 figure-eight $b_{--}$ orbit structure consisting of the point vortices $v_4$ and $v_5$ with the  strengths $\kappa_4=\kappa_5=-1$, and $a_2$ connecting the solid boundary $C_2$ with the two $\partial$-saddle connections around the point vortices $v_6$ and $v_7$. Hence, we have 
 \begin{eqnarray*}
 \B_{as}^1 &=& a_2(\lambda_+, \lambda_-, \lambda_\sim) \cdot a_+(b_+(\sigma_+)) \cdot a_2(\lambda_+, \lambda_-, \lambda_\sim)  \\
  & & \cdot a_-(b_-(b_{--}\{ b_-(\sigma_-), b_-(\sigma_-) \})) \cdot a_2(c_+(b_+(\sigma_+), \lambda_-), c_-(b_-(\sigma_-), \lambda_+)).
 \end{eqnarray*}
 Plugging all symbols into the root structure $\beta_{\emptyset 2}$,  we obtain 
 \begin{eqnarray*}
 &&\beta_{\emptyset 2} (\{ c_+(b_+(\sigma_+)), \sigma_{\widetilde{-}0}, c_{2-}(\sigma_{\widetilde{-}0}, \sigma_{\widetilde{+}0}) \cdot c_-(b_-(\sigma_-)), \sigma_{\widetilde{+}0} \}, a_2 \cdot a_+(b_+(\sigma_+)) \\
 && \qquad \cdot a_2 \cdot a_-(b_-(b_{--}\{ b_-(\sigma_-), b_-(\sigma_-) \})) \cdot a_2(c_+(b_+(\sigma_+),) c_-(b_-(\sigma_-)))),
 \end{eqnarray*}
 in which we abbreviate symbols $\lambda_\pm$ and $\lambda_\sim$ to reduce its length.

 \subsection{An application to 3D vector fields}
  The present theory applies to a topological characterization for 3D flows. As a proof of concept, we introduce an example appearing in the joint work with Nippon
  Pneumatic MFG, Co., LTD. We are developing powder classifying devices by applying the present theory to the projection of a three-dimensional vector 
  field onto two-dimensional sections.   Figure~\ref{fig:NPK2}(a) is a 3D computational domain (the left panel) for the supersonic jet mill, PJM S-type~\cite{NPK},  which corresponds to the mill part of the device enclosed by the gray square of a side view in the middle panel. The mill is equipped with six nozzles from which compressed air is injected. Being stirred by compressed air, powders are classified by their size, and fine particles are absorbed in the suction port in the center. 
  See the top view of the right panel of Figure~\ref{fig:NPK2}(a). With 3D simulations of the Navier-Stokes flow in the computational domain, we project 
  an instantaneous vector field onto three planar sections, shown as dotted lines (b)--(d) in the left panel. We then construct orbit structures on the sections as shown in the left panels of Figure~\ref{fig:NPK2}(b)--(d).
 
The middle panel of Figure~\ref{fig:NPK2}(b) shows an ss-saddle connection diagram for the orbit structure in the left panel on the upper side section of the computational domain.
 Since the flow domain is bounded by a solid boundary along which the flow  direction is counter-clockwise, the root structure is $\beta_{\emptyset -}$. No class-$c_\pm$ structures
  are attached to the  boundary, the flow domain is filled with counter-clockwise periodic orbits, and its inner boundary is a cycle. This gives rise to the COT representation
 of the root structure, $\beta_{\emptyset -} ( b_+(p_{\widetilde{+}}(\B^1_{b\widetilde{+}})), \{ \lambda_{-} \})$.
The flow domain inside the cycle is a repelling transverse annulus, that is to say $\B^1_{b\widetilde{+}} = b_{\widetilde{+}}(\B^2_{\widetilde{+}}, \{ \B^2_{a\widetilde{+}s} \})$, which 
contains a sink without rotation, i.e., $\B^2_{\widetilde{+}} = \sigma_{\widetilde{+}0}$, and two slidable saddles enclosing counter-clockwise sinks
 $\B^2_{a\widetilde{+}s} = q_{\widetilde{+}} ( \sigma_{\widetilde{-}+}) \cdot q_{\widetilde{+}} (\sigma_{\widetilde{-}+})$. Hence, its COT representation becomes
 \[ 
\beta_{\emptyset -} ( b_+(p_{\widetilde{+}}(b_{\widetilde{+}}(\sigma_{\widetilde{+}0}, \{ q_{\widetilde{+}} (\sigma_{\widetilde{-}+}) \cdot q_{\widetilde{+}} (\sigma_{\widetilde{-}+}) \}))),  \{ \lambda_{-} \} ). 
\]
 Since the orbit structure contains one limit circuit, we divide the flow domain into two disjoint domains. By the metric completion, the outer domain is a periodic annulus 
 and the inner domain is equivalent to a spherical surface. By the boundary collapse, we have the linking structure represented as a surface graph 
 in the right panel of Figure~\ref{fig:NPK2}(b).

 Figure~\ref{fig:NPK2}(c) shows the orbit structure (the left panel) and its ss-saddle connection diagram (the middle panel).
  Since the compressed air is injected from the six nozzles in this section,  the flow structures are approximated by slidable $\partial$-saddles.
  The flow domain consists of an annular domain and a disk domain separated by the suction port in the center. Since the outer boundary of
  the annular domain has the six slidable $\partial$-saddles,  its root structure is represented by 
$\beta_{\emptyset 2} (\{  \B_{c+s}^1, \B_{\widetilde{-}}^1, \B_{c-s}^1, \B_{\widetilde{+}}^1,\B_{\gamma + s}^1 \}, \B_{as}^1 )$.
 The source and sink structures of the slidable saddles are a source without rotation, $\sigma_{\widetilde{+}0}$, and a solid boundary associated with
 a counter-clockwise flow $\beta_+\{ \lambda_+\}$. Picking up one of the slidable saddles as a special source/sink pair for the root structure, we regard the remaining five slidable saddles as class-$\gamma_+$ orbit structures. Accordingly,  $\B^1_{\widetilde{+}}=\sigma_{\widetilde{+}0}$, $\B^1_{\widetilde{-}}=\beta_-\{ \lambda_- \}$ and
$\B^1_{\gamma + s} = \gamma_{\widetilde{+} +}(\sigma_{\widetilde{+}0})\cdot \cdots \cdot \gamma_{\widetilde{+} +}(\sigma_{\widetilde{+}0}) := \gamma^{\cdot 5}_{\widetilde{+}+}(\sigma_{\widetilde{+}0})$.
Moreover, there are no class-$c_\pm$ orbit structures attached to the outer boundary and no class-$a$ orbits between the source structure $\B^1_{\widetilde{+}}$ 
and the sink structure $\B^1_{\widetilde{-}}$. Hence, we have
\[
 \beta_{\emptyset 2}(\{\lambda_\sim, \beta_-\{ \lambda_- \},  \lambda_\sim,  \sigma_{\widetilde{+}0}, \gamma^{\cdot 5}_{\widetilde{+}+}(\sigma_{\widetilde{+}0}) \}, \lambda_\sim).
  \]
 In the inner disk domain enclosed by the solid boundary, the root structure is $\beta_{\emptyset -}$. Since the flow domain is a repelling transverse annulus
  and the solid boundary has no class-$c_\pm$ orbit structures, its COT representation
 starts with $\beta_{\emptyset -} ( b_{\widetilde{+}}(\B^1_{\widetilde{+}}, \{ \B^1_{a\widetilde{+}s} \}),  \{ \lambda_{-} \} )$. The inner orbit structures are a source without rotation, 
and three self-connected saddle separatrices enclosing counter-clockwise periodic orbits and centers,  namely $\B^1_{\widetilde{+}} = \sigma_{\widetilde{+}0}$ and 
$\B^1_{a\widetilde{+}s} = a_+(b_+(\sigma_+)) \cdot a_+(b_+(\sigma_+)) \cdot a_+(b_+(\sigma_+)) := a_+(b_+(\sigma_+))^{ \cdot 3}$, which yields
\[
\beta_{\emptyset -} ( b_{\widetilde{+}}(\sigma_{\widetilde{+}0}, \{ a_+(b_+(\sigma_+))^{ \cdot 3} \}), \{ \lambda_{-}\}).
 \]
The linking structures on a disk and a sphere are obtained by the metric completion and the boundary collapse. See the right panel of Figure~\ref{fig:NPK2}(c).

The flow domain of Figure~\ref{fig:NPK2}(d) also consists of two disjoint domains separated by the suction port.  The orbit structure and the ss-saddle connection diagram are shown in the left and the middle panels respectively. The outer annular domain enclosed by the solid boundaries is filled with
 counter-clockwise periodic orbits. Hence, the COT representation is $\beta_{\emptyset -}(b_{+}(\beta_{+}\{ \lambda_+ \}))$. The inner flow domain is a disk domain enclosed by
the solid boundary without source/sink structures and class-$c_\pm$ orbit structures. Since the domain is filled with non-closed orbits, the root structure becomes
$\beta_{\emptyset -} ( b_{\widetilde{+}}(\B^1_{\widetilde{+}}, \{ \B^1_{a\widetilde{+}s} \}),  \{ \lambda_{-} \})$.
As inner orbit structures, there exist a source without rotation, a self-connected saddle separatrix and a slidable saddle, namely
 $\B^1_{\widetilde{+}} = \sigma_{\widetilde{+}0}$, $\B^1_{a\widetilde{+}s} = a_+(\B^2_{b+}) \cdot a_{\widetilde{+}}\{ \B^2_{\widetilde{+}},  \B^3_{\widetilde{+}}\}$, which yields
\[ 
\beta_{\emptyset -} ( b_{\widetilde{+}}(\sigma_{\widetilde{+}0}, \{ a_+(\B^2_{b+}) \cdot a_{\widetilde{+}}\{ \B^2_{\widetilde{+}},  \B^3_{\widetilde{+}}\} \}), \{ \lambda_{-}  \}).
\]
The self-connected saddle separatrix encloses counter-clockwise periodic annulus around a center, $\B^2_{b+} = b_+(\sigma_+)$, and the slidable saddle exists between
a counter clockwise source, $\B^2_{\widetilde{+}} = \sigma_{\widetilde{+}+}$ and the source without rotation, $\B^3_{\widetilde{+}} = \sigma_{\widetilde{+}0}$. Hence,
the COT representation finally becomes
\[
 \beta_{\emptyset -} ( b_{\widetilde{+}}(\sigma_{\widetilde{+}0}, \{ a_+(b_+(\sigma_+)) \cdot a_{\widetilde{+}}\{ \sigma_{\widetilde{+}+}, \sigma_{\widetilde{+}0} \} \}),  \{\lambda_{-}\} ).
 \]
Their linking structures are shown in the right panel of Figure~\ref{fig:NPK2}(d).

\begin{figure}
\begin{center}
\includegraphics[scale=0.55]{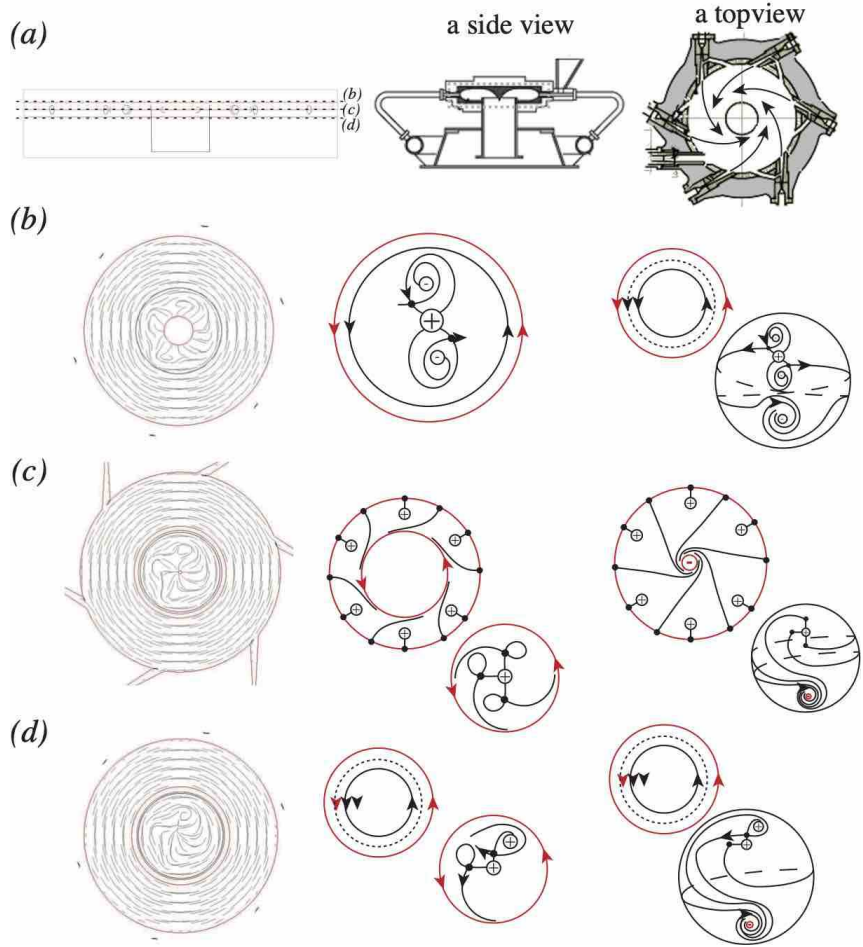}
\end{center}
\caption{(a) Left: a three-dimensional computational flow domain for the supersonic jet mill PJM S-Type, Nippon Pneumatic MFG. Co., LTD. Middle: a side view. The computational domain to the right corresponds to the part of the device enclosed by the gray rectangle. Right: a top view. The pictures are reprinted from the webpage by courtesy of the company.
In the left panel of each figure (b-d), we show orbit structures constructed by projecting three-dimensional vector fields on the sections sliced
  at (b) the upper part, (c) the middle part with six nozzles, and (d) the lower part of the computational domain.  Their corresponding ss-saddle connection diagrams and linking structures are shown in the middle and right panels, respectively.}
\label{fig:NPK2}
\end{figure}

\subsection{Application to interdisciplinary research}
The COT representations of the flows of finite type are used not only to distinguish the topological structures of flow patterns mathematically. Since the COT representation of the flow pattern is always associated with a decomposition of the flow domain separated by border orbits,  we can quantify the geometric information on the flow domains such as areas and geometric centers by using them.

 In what follows, to show the usefulness of our method as a topological data analysis, we introduce the application to practical research and development of a new powder classification device 
 conducted with Nippon Pneumatic MFG Co.Ltd, and we explain how the COT representations were used. The goal of the development is to optimize the shape of the device in which high-speed airflow classifies 
 powders accurately and efficiently by particle size.  In this optimization problem, we use the COT representations in the following way:  We obtain a two-dimensional section of the three-dimensional incompressible
  flow in the device that is computed numerically. Figure~\ref{fig:NPK3}(a) is a visualization of the numerical simulation of the flow pattern, from which we extract border orbits as in Figure~\ref{fig:NPK3}(b). In this figure, since the flow domain
  has an outflow opening to the left of the flow domain and an inflow small opening at the right-bottom, we add  $\partial$-saddles corresponding to the outflow and the inflow orbit structures, and we close the boundary of the 
  flow domain as shown in Figure~\ref{fig:NPK3}(c). This gives rise to the ss-saddle connection diagram of the flow of finite type to which we provide the following COT representation.

\[
\beta_{\emptyset 2}(\{a_{\widetilde{+}}\{\sigma_{\widetilde{+}-},\sigma^1_{\widetilde{+}+}\}, a_{\widetilde{-}}\{\sigma^2_{\widetilde{-}+},a_{\widetilde{-}}\{\sigma^3_{\widetilde{-}+},\{a_{\widetilde{-}}\{\sigma_{\widetilde{-}-},\sigma_{\widetilde{-}0}\}\}\}, \gamma_{\widetilde{+}+}(a_{\widetilde{+}}\{\sigma^4_{\widetilde{+}-},\sigma_{\widetilde{+}0}\})\}).
\]

Note that the COT representation and the linking structure in Figure~\ref{fig:NPK3}(d) are uniquely assigned to this flow pattern.
The flow domain is then decomposed into several subdomains separated by the border orbits, and the COT representation is available as the symbolic identifier of these border orbits. 
In particular, a downflow region is expressed by the gray region in Figure~\ref{fig:NPK3}(d) that is surrounded by the ss-components connecting the slidable saddles with the sink/source structures, 
$\sigma^1_{\widetilde{+}-}$, $\sigma^2_{\widetilde{-}+}$, $\sigma^3_{\widetilde{-}+}$ and $\sigma^4_{\widetilde{+}-}$.  It is known that the position and the area of this downflow region play an important role in 
the efficiency of the device. Hence, we compute the area of the gray downflow region that is extracted based on the COT representation, and we then modify the shape of the device so that the downflow region gets 
larger. Repeating this process, we found an optimal shape of the device in silico, based on which we produced a prototype device, and we finally brought a new product of Figure~\ref{fig:NPK3}(f) 
to the market. See the catalog of the product~\cite{NPK2} and the press release~\cite{JST} for the detail.  In this application, quantifying the downflow region in terms of the COT representations was the key to the successful development of the device, which was 
brought by the proposed topological data analysis.

\begin{figure}
\begin{center}
\includegraphics[scale=0.35]{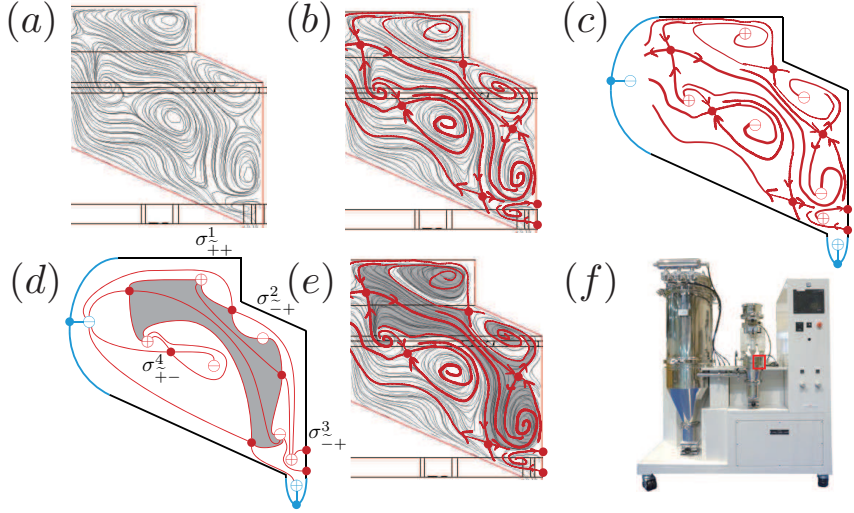}
\end{center}
\caption{
(a) A flow pattern on a 2D section of a three-dimensional incompressible flow. (b) Border orbits extracted from the flow pattern. (c) The border orbits compensated by $\partial$-saddles at the boundary
 in the flow domain. (d)  The ss-saddle connection diagram and the linking structure of this flow pattern. (e) The gray region represents a downflow region extracted from the COT representation. (f) ULTRA FINE SEPARATOR:Cnine, Nippon Pneumatic MFG. Co., LTD has been developed through the topological data analysis proposed in this paper.}
\label{fig:NPK3}
\end{figure}

\section{{Summary}}
We have shown that the topological structure of particle orbits generated by any flow of finite type corresponds uniquely to the pair of a partially cyclically ordered rooted tree (COT) and a set of directed surface graphs representing the linking structure. The discrete representations are utilized as simple symbolic identifiers for the topological structure of complex flows. Since the flow of finite type contains structurally stable 2D Hamiltonian flows, the present theory is an extension of the existing results on the incompressible vector fields~\cite{SY18,YS12,UYS18}.  Owing to this extension, one can extract discrete structures represented by trees/graphs and symbolic COT expressions from compressible 2D vector fields. From a theoretical point of view, it is important to note that the flow of finite type contains Morse--Smale vector fields, which are dense in the set of surface flows. The classification of streamline topologies for Morse--Smale flows has been considered in \cite{KMP,OS,Peixoto} to which our theory would make contributions in terms of discrete combinatorics. In addition,  we can apply the theory to the projection or the restriction of 3D vector fields on 2D sections. The topological classification of 3D vector fields is generally difficult since simple mathematical classification theories such as Poincar\'e--Bendixson theorem for 2D vector fields are not available. However, the projection method possibly enables us to classify and identify the topological structures of 3D vector fields as layered 2D vector fields as demonstrated in Section~5.2.  

The COT representations of the flow of finite type are more than just symbolic identifiers of the topological structures of flow patterns. They apply to the industrial problem as we have observed
 in Section~5.3. In particular, when we would apply the topological data analysis to a big dataset of flow patterns such as snapshots obtained by a long time evolution of fluid flows, we could 
 understand the dynamics of the flow pattern in terms of topology.  The proposed method of topological data analysis is called the topological flow data analysis (TFDA). 
 
 \section{{Future works}}
We discuss future works of the present study from three aspects: theoretical studies, algorithm development, and applications to real problems.
\paragraph{Theoretical studies:} 
{ We have several directions to extend the present topological classification of orbit structures. One direction is adding degenerate singular orbits to the flow structure since
 they can express inlets and outlets appearing in various flows such as cardiovascular flows, environmental flows, and industrial equipment. Multi-saddle connections from such degenerate singular orbits
  for Hamiltonian vector fields on closed surfaces are studied from the viewpoint of integrable systems~\cite{BoFo04}. Based on this, we could extend the present classification theory of the flow of finite type to that
  with degenerate singular orbits.  The other direction is to consider the topological classification of orbit structures in domains satisfying doubly periodic boundary conditions. This extension is useful since a doubly periodic
  boundary condition is usually assumed in the numerical studies of complex flows such as 2D isotropic turbulence. 
 
In the meantime, we can develop a new method of data analysis by combining the topological flow data analysis with data science techniques. For instance, using  COT representations as 
annotations for flow pattern images, we can realize an efficient flow pattern recognition with deep learning. We can also conduct the clustering analysis of a long-time evolution of fluid patterns in terms of COT 
representations, thereby revealing common geometric properties in the same cluster.  In addition, by tracking the transitions of the flow patterns among these clusters, we could obtain a 
data-driven Markov model elucidating the hidden dynamics of flow in the dataset.

\paragraph{Algorithm developments:} As described in the introduction, computer software converting given datasets into discrete 
 tree expressions associated with COT representations has been developed for structurally stable Hamiltonian vector fields. It plays a vital role in extending research fields to which the topological flow data analysis is applied. In a similar manner, it is desired to develop an efficient algorithm converting 
 a given flow of finite type into a COT representation and a linking structure. For the case of Hamiltonian flows, we have utilized persistent homology and Reeb graphs 
 to identify border orbits from the Hamiltonian in the algorithm. On the other hand, we cannot use the same idea for the flows of finite type since the Hamiltonian function no longer exists. Hence we need to develop an algorithm to extract border orbits directly from vector fields of the flow of finite type and implement it as computer software. 
  }

\paragraph{Applications to real problems:} It is important to expand the applicability of the TFDA to many research fields. Currently, with the support of Japan Science Technology (JST), 
we are promoting a research project to expand the application of TFDA~\cite{JST2}. In particular, we target their application to cardiovascular flows. Using the TFDA, we construct a theory to assign a COT representation to a 2D cross-sectional flow pattern of blood flow in the left ventricle that is obtained by cardiac echo and MRI. In this project, we will propose a new method for diagnosing cardiovascular disease in cooperation with doctors and hospitals. In addition, we will conduct the TFDA for data in a wide range of fields such as the environment, materials, life sciences, and industry to promote this 
analysis technology. The results of this research project will be published in the future.

%


\begin{thebibliography}{0}
\bibitem{andronov1937rough}
{A. A. Andronov and L. S. Pontryagin, 
Rough systems, 
{\it Dokl. Akad. Nauk SSSR}, {\bf 14} (1937), 247--250.}
\bibitem{ArBr98} 
H. Aref and M. Br{\o}ns, 
On stagnation points and streamline topology in vortex flows,  
{\it J. Fluid Mech}, {\bf  370} (1998), 1--27.
\bibitem{BoFo04} 
A. V. Bolsinov and A. T. Fomenko, 
{\it Integrable Hamiltonian Systems: Geometry, Topology, Classification},  
CRC Press, 2005.
\bibitem{BN97} 
{I. Bronstein and I. Nikolaev, 
Peixoto graphs of Morse-Smale foliations on surfaces, 
{\bf Topology Appl.}, {\bf 77} (1997), 19--36.}
\bibitem{DGKM16}
D.~G.~Crowdy, C.~C.~Green, E.~H.~Kropf, M.~M.~S.~Nasser, 
The Schottky-Klein prime function: a theoretical and computational tool for applications,  
{\it IMA J.  Appl. Math.}, {\bf  81} (2016), 589--628.
\bibitem{DC05}
D.~Crowdy and J.~Marshall, 
Analytical formulae for the Kirchhoff-Routh function in multiply connected domains,  
{\it Proc. Roy. Soc. A}, {\bf  461} (2005), 2477--2501.
\bibitem{DC11}
D.~Crowdy, 
Analytical formulae for  source and sink flows in multiply connected domains,  
{\it Theor. Comput. Fluid Dyn.}, {\bf  27} (2013), 1--19.

\bibitem{JST}
Japan Science and Technology (JST) Press release, 
{\it R\&D of powder classification devices with topological flow data analysis},  
See at \url{https://www.jst.go.jp/pr/announce/20190930-2/index.html} (in Japanese).
\bibitem{JST2}
Research Project ``Four-Dimensional Topological Data Analysis for Future Medical Care'', funded by Japan Science and Technology (JST).
See at \url{https://www.jst.go.jp/mirai/en/program/core/JPMJMI22G1.html}
\bibitem{HH}
G.~Hector and  U.~Hirsch, 
{\it Introduction to the geometry of foliations. Part B. Foliations of codimension one},  
Second ed. Aspects of Mathematics, E3. Friedr. Vieweg \& Sohn, Braunschweig, 1983.
\bibitem{KiNe00}
R. Kidambi and P.~K. Newton, 
Streamline topologies for integrable vortex motion on a sphere,  
{\it Phys. D}, {\bf  140} (2000), 95--125.
\bibitem{KMP}
{V.~Kruglov, D.~Malyshev  and O.~Pochinka, 
Topological classification of $\Omega$-stable flows on surfaces by means of effectively distinguishable multigraphs,  
{\it Disc. Cont. Dyn. Sys. A,}, {\bf  38} (2018), 4305--4327.}
\bibitem{KMP2018}
V.~Kruglov, D.~Malyshev  and O.~Pochinka, 
On algorithms that effectively distinguish gradient-like dynamics on surfaces, 
{\it Arnold Math. J.}, {\bf 4} (2018), no. 3-4, 483--504. 
\bibitem{LaPa90}
R. Labarca and M.J. Pacifico, 
Stability of Morse--Smale vector fields on manifolds with boundary,  
{\it Topology}, {\bf  29} (1990), 57--81.
\bibitem{MaWa05}
T. Ma and S. Wang, 
{\it Geometric theory of incompressible flows with applications to fluid dynamics},  
Mathematical Surveys and Monographs, 119, AMS, Providence, RI, 2005.
\bibitem{Mai43}
A. Maier, 
Trajectories on the closed orientable surfaces (Russian, with English summary), 
{\it Rec. Math. [Mat. Sbornik] N.S.}, {\bf  12} (1943), 71--84. 
\bibitem{Moffatt01}
K. Moffatt, 
The topology of scalar fields in 2D and 3D turbulence, 
in {\it IUTAM symposium on geometry and statistics of turbulence} (eds. T. Kambe, T. Nakano, T. Miyauchi). Springer, Dordrecht, (2001), 13--22. 
\bibitem{NPK}
NPK Co., LTD, PJM, \url{http://www.npk.jp/english/chemical\_e/products/pjm/featuer.html} 
\bibitem{NPK2}
NPK Co., LTD, ULTRA FINE SEPARATOR: Cnine, whose catalogue is available at \url{https://www.npk-en.com/products/chemical-engineering-catalogue/}
\bibitem{Nikolaenko20}
{S. S. Nikolaenko. 
Topological classification of Hamiltonian systems on two-dimensional non-compact manifolds, 
{\bf Sbornik Mathematics}, {\bf 211}(8) (2020), 1127--1158.}
\bibitem{NZ}
I.~Nikolaev and E.~Zhuzhoma, 
{\it Flows on 2-Dimensional Manifolds}, 
Lecture Notes in Mathematics, 1705,  Springer-Verlag, Berlin, 1999. 
\bibitem{nikolaev2001foliations}
{I. Nikolaev, 
{\it Foliations on Surfaces}, 
A Series of Modern Surveys in Mathematics, 41. Springer-Verlag, Berlin, 2001.}
\bibitem{OS}
A.~Oshemkov and V. Sharko, 
Classification of Morse--Smale flows on two-dimensional manifolds, 
{\it Mat. sb.}, {\bf  189} (1998), 93--140. 
 \bibitem{Peixoto}
 M.~Peixoto, 
Structural stability on two-dimensional manifolds, 
{\it Topology}, {\bf  1} (1962), 101--120 
 \bibitem{Prishlyak2020}
{O. O. Prishlyak, A. A. Prus, 
A three-color graph of Morse flow and a compact surface with a boundary, 
{\it Nel\={\i}n\={\i}\u{\i}n\={\i} Koliv.}, {\bf 22} (2019), no. 2, 250--261; translation in {\it J. Math. Sci.} {\bf 249} (2020), no. 4, 661--672.}
 \bibitem{SSY14}
T. Sakajo, Y. Sawamura and T. Yokoyama, 
Unique encoding for streamline topologies of incompressible and inviscid flows in multiply connected domains, 
{\it Fluid Dyn. Res.}, {\bf  46} (2014), 031411. 
\bibitem{SY15}
T.~Sakajo and T.~Yokoyama, 
Transitions between streamline topologies of structurally stable Hamiltonian flows in multiply connected domains, 
{\it Phys. D}, {\bf  307} (2015),  22--41. 
\bibitem{SY18}
T.~Sakajo and T.~Yokoyama, 
Tree representation of topological streamline patterns of structurally stable 2D Hamiltonian vector fields in multiply connected domains, 
{\it IMA J.  Appl. Math.}, {\bf  83} (2018), 380--411. 
\bibitem{SOU22}
{T.~Sakajo, S.~Ohishi, and T.~Uda,
Identification of Kuroshio meanderings south of Japan via a topological data analysis for sea surface height,
{\it J. Ocean.} (2022) in press.}
\bibitem{S61}
S.~Smale, 
On gradient dynamical systems, 
{\it Ann. of Math.}, {\bf  74} (1961), 199--206. 
\bibitem{YS12}
T.~Yokoyama and T.~Sakajo, 
Word representation of streamline topologies for structurally stable vortex flows in multiply connected domains, 
{\it Proc. Roy. Soc. A,}, {\bf  469} (2013), 20120558. 
\bibitem{Y}
T.~Yokoyama, 
Decompositions of surface flows, preprint,
arXiv:1703.05501.
\bibitem{YY21}
{T.~Yokoyama, T.~Yokoyama, 
Complete transition diagrams of generic Hamiltonian flows with a few heteroclinic orbits, 
{\it Discrete Math. Algorithms Appl.}, {\bf13} (2021), no. 2, 2150023.}
\bibitem{Y21}
{T.~Yokoyama, 
Topological characterizations of Hamiltonian flows with finitely many singular points on unbounded surfaces, preprint, 
arXiv:2111.01420.}
\bibitem{UYS18}
T.~Uda, T.~Yokoyama and T.~Sakajo, 
Algorithms converting streamline topologies for 2D Hamiltonian vector fields using Reeb graphs and persistent homology, 
{\it Trans. Japan Soc. Indust. Appl. Math.}, {\bf  29} (2019),  187--224 (in Japanese). 
\bibitem{USIK20} 
T.~Uda, T. Sakajo, M. Inatsu and K. Koga,
Morphological identification of atmospheric blockings by topological flow data analysis.
{\it J. Meteo. Soc. Japan} {\bf  99} (2021) (doi:10.2151/jmsj.2021-057).
\bibitem{wang1990c}
{X. Wang, 
The $C^*$-algebras of Morse-Smale flows on two-manifolds, 
{\it Ergodic Theory Dynam. Systems}, {\bf 10} (1990), 565--597.}

\end{thebibliography}
\end{document}